\pdfoutput=1
\RequirePackage{ifpdf}
\ifpdf % We are running pdfTeX in pdf mode
\documentclass[pdftex]{sigma}
\else
\documentclass{sigma}
\fi

\numberwithin{equation}{section}

\newtheorem{Theorem}{Theorem}[section]
\newtheorem{Corollary}[Theorem]{Corollary}
\newtheorem{Lemma}[Theorem]{Lemma}
\newtheorem{Proposition}[Theorem]{Proposition}
 { \theoremstyle{definition}
\newtheorem{Definition}[Theorem]{Definition}

 }

\def\R{\mathbb{R} }
\def\E{\mathbb{E} }
\def\Z{\mathbb{Z} }
\def\N{\mathbb{N} }
\def\C{\mathbb{C} }

\def\K{\mathbb{K} }

\DeclareMathOperator{\TKK}{TKK}
\DeclareMathOperator{\ad}{ad}
\DeclareMathOperator{\Inn}{Inn}
\DeclareMathOperator{\End}{End}
\DeclareMathOperator{\bessel}{\mathcal B}
\DeclareMathOperator{\mbessel}{\widetilde{\mathcal B}}

\DeclareMathOperator{\SB}{SB}
\DeclareMathOperator{\B}{\mathbb I}

\DeclareMathOperator{\Fock}{\mc F}
\DeclareMathOperator{\gsh}{\widetilde {\mc H}}
\newcommand{\iSB}{\SB^{-1}}

\newcommand{\bfip}[1]{\langle {#1}\rangle _\mathcal B}

\newcommand{\bfipbar}[1]{\overline{\langle {#1}\rangle }_\mathcal B}
\newcommand{\ip}[1]{\langle {#1}\rangle _W}
\newcommand{\ipbar}[1]{\overline{\langle {#1}\rangle }_W}
\newcommand{\bbf}[1]{\langle {#1}\rangle _\beta}
\newcommand{\pt}[1]{\partial_{#1}}
\newcommand{\mf}[1]{\mathfrak{#1}}
\newcommand{\mc}[1]{\mathcal{#1}}

\begin{document}
\allowdisplaybreaks

\newcommand{\arXivNumber}{2002.12836}

\renewcommand{\PaperNumber}{085}

\FirstPageHeading

\ShortArticleName{A Fock Model and the Segal--Bargmann Transform for $\mathfrak{osp}(m,2|2n)$}

\ArticleName{A Fock Model and the Segal--Bargmann Transform for the Minimal Representation of the Orthosymplectic Lie Superalgebra $\boldsymbol{\mathfrak{osp}(m,2|2n)}$}

\Author{Sigiswald BARBIER, Sam CLAEREBOUT and Hendrik DE BIE}
\AuthorNameForHeading{S.~Barbier, S.~Claerebout and H.~De Bie}
\Address{Department of Electronics and Information Systems, Faculty of Engineering and Architecture,\\ Ghent University, Krijgslaan 281, 9000 Gent, Belgium}
\Email{\href{mailto:Sigiswald.Barbier@UGent.be}{Sigiswald.Barbier@UGent.be}, \href{mailto:Sam.Claerebout@UGent.be}{Sam.Claerebout@UGent.be}, \href{mailto:Hendrik.DeBie@UGent.be}{Hendrik.DeBie@UGent.be}}

\ArticleDates{Received March 17, 2020, in final form August 12, 2020; Published online August 26, 2020}

\Abstract{The minimal representation of a semisimple Lie group is a `small' infinite-di\-mensional irreducible unitary representation. It is thought to correspond to the minimal nilpotent coadjoint orbit in Kirillov's orbit philosophy. The Segal--Bargmann transform is an intertwining integral transformation between two different models of the minimal representation for Hermitian Lie groups of tube type. In this paper we construct a Fock model for the minimal representation of the orthosymplectic Lie superalgebra $\mathfrak{osp}(m,2|2n)$. We also construct an integral transform which intertwines the Schr\"odinger model for the minimal representation of the orthosymplectic Lie superalgebra $\mathfrak{osp}(m,2|2n)$ with this new Fock model.}


\Classification{17B10; 17B60; 22E46; 58C50}

\section{Introduction}

The classical Segal--Bargmann transform
\begin{align*}
\SB f(z) := \exp\left(-\tfrac{1}{2} z^2\right) \int_{{\mathbb R}^m} \exp(2 x \cdot z) \exp\big({-}x^2\big)f(x) \mathrm{d} x
\end{align*}	
is a unitary isomorphism from the space of square integrable functions on $\mathbb{R}^m$ to the Fock space of entire functions on $\mathbb{C}^m$ which are square integrable which respect to the weight function $\exp\big({-}|z|^2\big)$.
The Segal--Bargmann transform is defined in such a way that it maps the creation (resp.\ anihilation) operators on the Schr\"odinger space to coordinate multiplication (resp. differentiation) on the Fock space. This implies in particular that the harmonic oscillator becomes the much simpler Euler operator on the Fock space~\cite{Bargmann}. The Segal--Bargmann transform can also be interpreted as an intertwining operator between two models of the metaplectic representation (also known as the Segal--Shale--Weil or oscillator representation) of the metaplectic group, a double cover of the symplectic group. See~\cite{Folland} for more on the classical Segal--Bargmann transform and the metaplectic representation.

The (even part of the) metaplectic representation is a prominent example of a minimal representation. Another extensively studied example is the minimal representation of ${\rm O}(p,q)$ \cite{KO1,KO2,KO3,KM}. The minimal representation of a semisimple Lie group is the irreducible unitary representation that according to the orbit method should correspond to the minimal nilpotent coadjoint orbit~\cite{GanSavin}. The Segal--Bargmann transform has been generalized to this setting of minimal representations. Namely for a Hermitian Lie group of tube type, there exists an explicit integral transform which intertwines the Schr\"odinger and Fock model of the minimal representation~\cite{HKMO}.

Lie superalgebras and Lie supergroups are generalisations of Lie algebras and Lie groups. They were introduced to mathematically describe supersymmetry. Their representation theory is an active field of study with still a lot of open questions, for instance a description of the unitary irreducible representations. Since most ingredients of the orbit method still exist in the super setting, it is believed that the orbit method should also in the super setting be a good tool for the study of irreducible representations \cite[Chapter~6.3]{Kirillov}. For example, the irreducible unitary representations of nilpotent Lie supergroups have been classified that way~\cite{NeebSalmasian, Salmasian}.

An ambitious aim in this light would be to construct minimal representations and the intertwining Segal--Bargmann transform for Lie supergroups. Recently a first step in that direction has already been taken. Namely the construction of a minimal representation of the orthosymplectic Lie supergroup ${\rm OSp}(p,q|2n)$ was accomplished in~\cite{BF}. In the bosonic case (i.e., $n=0$) this realisation corresponds to the Schr\"odinger model for ${\rm O}(p,q)$ of~\cite{KM}.

\looseness=1 In this article we achieve two further goals. First we construct a Fock model for the minimal representation of the Lie superalgebra $\mathfrak{osp}(m,2|2n)$. We also define an integral transform which intertwines the Schr\"odinger model of the minimal representation of $\mathfrak{osp}(m,2|2n)$ with this Fock model. Note that only for $q=2$ the Lie group ${\rm O}(p,q)$ is Hermitian of tube type and thus only in that case do we have a Segal--Bargmann transform. For that reason we have only constructed a Segal--Bargmann transform in the super setting for $\mathfrak{osp}(m,2|2n)$. Our main results hold for $m-2n \geq 4$. This restriction comes from \cite{BF}, where key properties of the integral we use in the definition of the Segal--Bargmann transform are only proven for the case \mbox{$m-2n \geq 4$}.

We will work in this paper always on the algebraic level. So we will work with representations of the Lie superalgebra $\mathfrak{osp}(m,2|2n)$ instead of the Lie supergroup ${\rm OSp}(m,2|2n)$, and we will act on super-vector spaces defined using superpolynomials instead of using global sections of a~supermanifold. This allows us to circumvent the delicate technicalities associated with supergroups and supermanifolds. Note that in the bosonic case, the spaces we work with are dense in certain Hilbert spaces. Using standard techniques one can then integrate the representation to group level and extend to the whole Hilbert space, see for example \cite[Theorem~2.30]{HKM} or \cite[Theorem~2.17]{HKMO}. These techniques no longer work/exist in the super case. There does exist an abstract way to integrate a so-called admissible $(\mathfrak{g},K)$-module to group level \cite{Alldridge} which was for example used in~\cite{BF} to integrate the minimal representation of $\mathfrak{osp}(p,q|2n)$ to ${\rm OSp}(p,q|2n)$. However, this approach is not very concrete.

Explicit examples such as the one constructed in this paper could help to develop such integration tools and to find the correct definitions in the super setting. For example our representations ought to be `unitary' in some sense. A definition of unitary representations does exists in the supersetting \cite[Definition~2]{CCTV}. However, a large class of Lie superalgebras, including $\mathfrak{osp}(p,q|2n)$, do not allow for any super unitary representation in this sense \cite[Theorem~6.2.1]{NeebSalmasian}. This highly unsatisfactory situation has inspired the search for a new or extended definition of a~unitary representation~\cite{dGM, Tuynman}. At the moment it is still unclear what the right definition should be, but we believe that the construction of explicit examples which ought to be `unitary' could be useful for this endeavour.

\subsection{Structure of the paper}
We structure the paper as follows. In Section~\ref{Preliminaries and notations} we fix notations and introduce the spaces and algebras we will use throughout the paper. In Section~\ref{The Schrodinger representation} we recall the Schr\"odinger model of the minimal representation of $\mathfrak{osp}(m,2|2n)$ defined in~\cite{BF}. We also introduce an integral which leads to an $\mathfrak{osp}(m,2|2n)$-invariant, non-degenerate, superhermitian form.

The next three sections contain the main body of this paper. In Section~\ref{The Fock space} we construct the Fock space as a quotient of the space of superpolynomials. We then define the Bessel--Fischer product, which gives us a non-degenerate, superhermitian form on our Fock space (Propositions~\ref{supersym} and~\ref{nondeg}).
In the bosonic case $(n=0)$, this Bessel--Fischer product is equivalent to the inner product coming from an integral on the Fock space \cite[Proposition 2.6]{HKMO}. Since we do no longer have this integral in the super setting, we construct a direct proof for the superhermitian property. This seems new even in the bosonic case.
We also show that our Fock space has a~reproducing kernel (Theorem \ref{Theorem repr kernel}).

In Section~\ref{The Fock model} we endow this Fock space with an $\mathfrak{osp}(m,2|2n)$-module structure leading to a~Fock model of the minimal representation of $\mathfrak{osp}(m,2|2n)$. We prove that this is an irreducible representation and obtain a~very explicit description (Theorem \ref{ThDecF}). In particular, we have complete branching rules for the subalgebras $\mathfrak{osp}(m|2n)$ and $\mathfrak{osp}(m-1|2n)$.

In Section \ref{The Segal--Bargmann transform}, we define an integral transform which maps the space of functions used in the Schr\"odinger realisation to the space of functions of the Fock realisation (Definition \ref{DefSB}). We show that this integral is an intertwining isomorphism which preserves the superhermitian form (Theorems \ref{ThIP} and \ref{PropUnitSB}). We also give an explicit inverse (Definition \ref{Def SB inv}).
As an application we use the Segal--Bargmann transform to define generalised Hermite functions.

In Appendix \ref{Special functions} we gather some definitions and results on special functions which are used throughout the paper. We have also put the technical and lengthy proof of Theorem \ref{ThIP} in Appendix \ref{AppIP}.

\section{Preliminaries and notations}
\label{Preliminaries and notations}
In this paper Jordan and Lie algebras will be defined over complex numbers $\C$ if they have a~$\C$ in subscript, otherwise they are defined over the field of real numbers $\mathbb R$. Function spaces will always be defined over the complex field~$\C$. We use the convention $\mathbb N = \{0,1,2,\ldots\}$.

A super-vector space is defined as a $\Z_2$-graded vector space, i.e., $V=V_{0}\oplus V_{1}$, with~$V_{0}$ and~$V_{1}$ vector spaces. An element $v$ of a super-vector space $V$ is called homogeneous if it belongs to~$V_i$, $i\in \Z_2$. We call~$i$ the parity of $v$ and denote it by $|v|$. An homogeneous element $v$ is even if $|v|=0$ and odd if $|v|=1$. When we use $|v|$ in a formula, we are considering homogeneous elements, with the implicit convention that the formula has to be extended linearly for arbitrary elements. We denote the super-vector space~$V$ with $V_{0} = \mathbb R^m$ and $V_{1} = \mathbb R^n$ as $\mathbb R^{m|n}$. We will always assume $m\geq 2$.

\subsection{Superpolynomials}

Let $\mathbb K$ be either $\mathbb R$ or $\C$.
\begin{Definition}
The \textit{space of superpolynomials} over $\mathbb K$ is defined as
\begin{gather*}
\mathcal P\big(\mathbb K^{m|2n}\big):=\mathcal P \big(\mathbb K^{m}\big)\otimes_\C\Lambda\big(\mathbb K^{2n}\big),
\end{gather*}
where $\mc P\big(\mathbb K^m\big)$ denotes the space of complex-valued polynomials over the field $\mathbb K$ in $m$ variables and $\Lambda\big(\mathbb K^{2n}\big)$ denotes the Grassmann algebra in $2n$ variables. The variables of $\mc P(\mathbb K^m)$ and $\Lambda\big(\mathbb K^{2n}\big)$ are called even and odd variables, respectively. They satisfy the commutation relations
\begin{gather*}
x_ix_j = (-1)^{|i||j|}x_j x_i,
\end{gather*}
where $|i| := |x_i|$, $i\in\{0, 1, \ldots, m+2n-1\}$.
\end{Definition}

Let $\bbf{\cdot\,,\cdot}$ be a supersymmetric, non-degenerate, even bilinear form on $\mathbb K^{m|2n}$ with basis $\{x_i\}_{i=0}^{m+2n-1}$. We denote the matrix components by $\beta_{ij} := \bbf{x_i,x_j}$ and denote the components of the inverse matrix by $\beta^{ij}$, i.e., $\beta^{ij}$ is defined such that $\sum_j \beta_{ij}\beta^{jk}=\delta_{ik}$. Set $x^j = \sum_i x_i \beta^{ij}$.
The differential operator $\partial^i$ is defined as the unique derivation in $\End\big(\mc P\big(\mathbb K^{m|2n}\big)\big)$ such that $\partial^i(x_j) = \delta_{ij}$, with $\delta_{ij}$ the Kronecker delta.
We also define $\pt {j} = \sum_i \partial^i \beta_{ji}$. Then it holds that $\pt i(x^j) = \delta_{ij}$.

When we are working with both real and complex polynomials at the same time, we will denote $\pt i$ and $\partial^i$ for the real variable $x_i$ as $\pt {x^i}$ and $\pt{x_i}$, respectively. Similarly, we will denote~$\pt i$ and~$\partial^i$ for the complex variable $z_i$ as $\pt {z^i}$ and $\pt{z_i}$, respectively.

We will make frequent use of the following operators:
\begin{gather}\label{EqR2ELap}
R^2 := \sum_{i,j} \beta^{ij}x_ix_j, \qquad \E := \sum_{i,j} \beta^{ij} x_i\pt j \qquad \mbox{and} \qquad \Delta := \sum_{ij}\beta^{ij}\pt i \pt j.
\end{gather}
Here, the operator $R^2$ is called the square of the radial coordinate and acts through multiplication. The operators $\E$ and $\Delta$ are called the Euler operator and the Laplacian, respectively. We have the following lemma.

\begin{Lemma}\label{Lemsltriple}
The operators $R^2$, $\E$ and $\Delta$ satisfy
\begin{gather*}
[\Delta, R^2] = 4\E+2M, \qquad [\Delta, \E] = 2\Delta, \qquad [R^2,\E] = -2R^2,
\end{gather*}
where $M=m-2n$ is the superdimension. In particular, $\big(R^2,\E+\frac{M}{2}, -\frac{\Delta}{2}\big)$ forms an $\mf{sl}_\K(2)$-triple.
Furthermore, they commute in $\End\big(\mc P\big(\K^{m|2n}\big)\big)$ with the operators \[ L_{ij} := x_i\pt j - (-1)^{|i||j|}x_j\pt i. \]
\end{Lemma}
\begin{proof}A straightforward calculation or see, for example, \cite{DS}.
\end{proof}

If we are working with two sets of variables we will add a variable indicator to avoid confusion. For example, we denote
\begin{gather*}
R^2_x = \sum_{i,j} \beta^{ij}x_ix_j, \qquad \E_x = \sum_{i,j} \beta^{ij} x_i\pt {x^j}, \qquad \Delta_x = \sum_{ij}\beta^{ij}\pt {x^i} \pt {x^j}
\end{gather*}
and $L_{ij}^x = x_i\pt {x^j} - (-1)^{|i||j|}x_j\pt {x^i}$ for the real variables and
\begin{gather*}
R^2_z = \sum_{i,j} \beta^{ij}z_iz_j, \qquad \E_z = \sum_{i,j} \beta^{ij} z_i\pt {z^j}, \qquad \Delta_z = \sum_{ij}\beta^{ij}\pt {z^i} \pt {z^j}
\end{gather*}
and $L^z_{ij} = z_i\pt {z^j} - (-1)^{|i||j|}z_j\pt {z^i}$ for the complex variables.

\subsection{The orthosymplectic Lie superalgebra}

Let $\mathbb K$ be either $\mathbb R$ or $\C$.
\begin{Definition}
The \textit{orthosymplectic Lie superalgebra} $\mf{osp}_{\mathbb K}(m|2n,\beta)$ is defined as the subalgebra of $\mf{gl}_{\mathbb K}(m|2n)$ preserving a supersymmetric non-degen\-erate even bilinear form $\beta$. Thus $\mf{osp}_{\mathbb K}(m|2n,\beta)$ is spanned by $X \in \mf{gl}_{\mathbb K}(m|2n) $ such that
\begin{gather*}
\bbf{X(u),v}+(-1)^{|u||X|}\bbf{u,X(v)} = 0,
\end{gather*}
for all $u,v \in \K^{m|2n}$.
\end{Definition}
The orthosymplectic Lie superalgebra has a differential operator realisation on $\mc P\big(\K^{m|2n}\big)$. A~basis in this realisation is given by
\begin{gather*}
L_{ij} := x_i\pt j - (-1)^{|i||j|}x_j\pt i, \qquad \mbox{for } i<j,\\
L_{ii} := 2x_i\pt i, \qquad \mbox{for } |i| = 1.
\end{gather*}

\subsection{Spherical harmonics}

Let $\mathbb K$ be either $\mathbb R$ or $\C$.
The space of homogeneous superpolynomials of degree $k$ is denoted by
\begin{gather*}
\mc P_k\big(\K^{m|2n}\big) := \big\{p\in\mc P\big(\K^{m|2n}\big)\colon \E p = k p\big\}.
\end{gather*}
The space of spherical harmonics of degree $k$ is defined by
\begin{gather*}
\mc H_k\big(\K^{m|2n}\big) := \big\{ f\in \mc P_k\big(\K^{m|2n}\big)\colon \Delta f = 0\big\},
\end{gather*}
i.e., it is the space of homogeneous polynomials of degree $k$ which are in the kernel of the Laplacian.

 The Fischer decomposition gives a decomposition of the space of superpolynomials using these spherical harmonics \cite[Theorem~3]{DS}.
\begin{Proposition}\label{LemFC}
If $m-2n\neq -2\N$, then $\mc P\big(\K^{m|2n}\big)$ decomposes as
\begin{gather*}
\mc P\big(\K^{m|2n}\big) = \bigoplus_{k=0}^\infty \mc P_k\big(\K^{m|2n}\big) = \bigoplus_{k=0}^\infty\bigoplus_{j=0}^\infty R^{2j}\mc H_k\big(\K^{m|2n}\big).
\end{gather*}
\end{Proposition}

In \cite{LS} the following generalisation of the Fischer decomposition was obtained that still holds for the exceptional case $M\in -2\N$.
\begin{Proposition}[generalised Fischer decomposition]\label{PropFDgen}
The superspace $\mc P\big(\K^{m|2n}\big)$ decomposes as
\begin{gather*}
\mc P\big(\K^{m|2n}\big) = \bigoplus_{k=0}^\infty \mc P_k\big(\K^{m|2n}\big) = \bigoplus_{k=0}^\infty\bigoplus_{j=0}^\infty \big(R^{2}\Delta R^{2}\big)^j\gsh_k\big(\K^{m|2n}\big),
\end{gather*}
where
\begin{gather*}
\gsh_k\big(\K^{m|2n}\big) = \big\{ f\in \mc P_k\big(\K^{m|2n}\big)\colon \Delta R^2\Delta f = 0\big\}
\end{gather*}
is the space of generalised spherical harmonics of degree~$k$.
\end{Proposition}

From \cite[Theorem 5.2]{K2} we obtain the following.
\begin{Proposition}\label{PropIrred}
If $M\not\in -2\N$, then $\mc H_k\big(\K^{m|2n}\big)$ is an irreducible $\mf{osp}_\K(m|2n)$-module.
\end{Proposition}

The dimension of the spherical harmonics of degree $k$ is given in \cite[Corollary~1]{DS}.
\begin{Proposition}\label{PropDimH}
The dimension of $\mc H_k\big(\K^{m|2n}\big)$, for $m\neq 0$ is given by
\begin{gather*}
\dim \mc H_k\big(\K^{m|2n}\big) = \dim \mc P_k\big(\K^{m|2n}\big) - \dim \mc P_{k-2}\big(\K^{m|2n}\big),
\end{gather*}
with
\begin{gather*}
\dim \mc P_k\big(\K^{m|2n}\big) = \sum_{i=0}^{\min(k,2n)}\binom{2n}{i}\binom{k-i+m-1}{m-1}.
\end{gather*}
\end{Proposition}

We will also use the following formula for the dimension of the spherical harmonics of degree~$k$.
\begin{Proposition}\label{Prop dimension}
The dimension of $\mc H_k\big(\K^{m|2n}\big)$, for $m > 1$ is given by
\begin{gather*}
\dim \mc H_k\big(\K^{m|2n}\big) = \dim \mc P_k\big(\K^{m-1|2n}\big) + \dim \mc P_{k-1}\big(\K^{m-1|2n}\big).
\end{gather*}
\end{Proposition}
\begin{proof}
If we prove
\begin{gather*}
\dim \mathcal P_k\big(\K^{m|2n}\big) - \dim \mathcal P_{k-2}\big(\K^{m|2n}\big) - \dim \mathcal{P}_k\big(\K^{m-1|2n}\big) - \dim\mathcal{P}_{k-1}\big(\K^{m-1|2n}\big)=0,
\end{gather*}
then the proposition follows from Proposition~\ref{PropDimH}. First suppose $2n\leq k-2$, then the above equation becomes
\begin{gather*}
\sum_{i=0}^{2n} \binom{2n}{i} \left(\binom{k-i+m-1}{m-1}-\binom{k-i+m-3}{m-1}\right.\\
 \left.\qquad{}-\binom{k-i+m-2}{m-2}-\binom{k-i+m-3}{m-2}\right) = 0,
\end{gather*}
which is true since the recursive formula of binomial coefficients gives us that
\begin{gather*}
\binom{k-i+m-1}{m-1}-\binom{k-i+m-3}{m-1}-\binom{k-i+m-2}{m-2}-\binom{k-i+m-3}{m-2} = 0,
\end{gather*}
for all $i\in\{0, \ldots, 2n\}$. For $2n > k-2$ we have the following extra terms
\begin{gather*}
\binom{2n}{k-1}\left(\binom{m}{m-1}-\binom{m-1}{m-2}-\binom{m-2}{m-2}\right) + \binom{2n}{k}\left(\binom{m-1}{m-1}-\binom{m-2}{m-2}\right),
\end{gather*}
where we ignore the last term if $2n=k-1$. Using basic binomial properties, these terms are clearly equal to zero.
\end{proof}

For $n = 0$ a more insightful reasoning as to why this formula holds is given by Proposition 5 in \cite{DBGvdVV}.

\subsection[The spin factor Jordan superalgebra $J$]{The spin factor Jordan superalgebra $\boldsymbol{J}$}
To each Hermitian Lie group of tube type corresponds an Euclidean Jordan algebra. These Jordan algebras were a crucial ingredient in the unified approach to construct minimal representations and the Segal--Bargmann transform \cite{HKM, HKMO}. More concretely, one can associate with each Jordan algebra certain Lie algebras (the structure algebra and the TKK-algebra), and the Jordan algebra is also used in the construction of spaces on which these Lie algebras act. In this paper we will not use anything directly from Jordan theory, but introducing $\mathfrak{osp}(m,2|2n)$ via the spin factor Jordan superalgebra leads to a natural decomposition of $\mathfrak{osp}(m,2|2n)$ as well as to some interesting subalgebras that we will use.
\begin{Definition}
A \textit{Jordan superalgebra} is a supercommutative superalgebra $J$ satisfying the Jordan identity
\begin{gather*}
(-1)^{|x||z|}[L_{x}, L_{yz}]+(-1)^{|y||x|}[L_{y}, L_{zx}]+(-1)^{|z||y|}[L_{z}, L_{xy}]=0 \qquad \text{for all } x,y,z \in J.
 \end{gather*}
Here the operator $L_x$ is (left) multiplication with $x$ and $[\cdot\,,\cdot]$ is the supercommutator, i.e., $[L_x,L_y] := L_xL_y - (-1)^{|x||y|}L_yL_x$.
\end{Definition}

Let $\mathbb K$ be either $\mathbb R$ or $\C$. We will define the spin factor Jordan superalgebra associated with a supersymmetric, non-degenerate, even, bilinear form. Let $V_\K$ be a super-vector space over $\mathbb K$ with $\dim(V_\K) = (m-1|2n)$ and a supersymmetric, non-degenerate, even, bilinear form $\langle \cdot \;, \cdot \rangle _{\tilde\beta}$ where, for $\K=\mathbb R$, the even part has signature $(m-1,0)$. Recall that we always assume $m\geq 2$. We choose a homogeneous basis $(e_i)_{i=1}^{m+2n-1}$ of $V_\K$. For $u = \sum_i u^i e_i$ and $v=\sum_i v^i e_i$ we then have
\begin{gather*}
\langle u,v\rangle _{\tilde\beta} = \sum_{i,j} u^i {\tilde\beta}_{ij}v^j\qquad \mbox{with} \quad {\tilde\beta}_{ij} := \langle e_i,e_j\rangle _{\tilde\beta}.
\end{gather*}
\begin{Definition}
The \textit{spin factor Jordan superalgebra} is defined as $J_\K := \K e_0 \oplus V_\K$ with $|e_0| = 0$. The Jordan product is given by
\begin{gather*}
(\lambda e_0+u)(\mu e_0 + v) = (\lambda\mu +\langle u,v\rangle _{\tilde\beta})e_0 + \lambda v +\mu u,
\end{gather*}
for $u,v\in V_\K$ and $\lambda,\mu \in \K$.
\end{Definition}
Thus $e_0$ is the unit of $J_\K$.
We extend the homogeneous basis $(e_i)_{i=1}^{m+2n-1}$ of $V_\K$ to a homogeneous basis $(e_i)_{i=0}^{m+2n-1}$ of $J_\K$ and extend the bilinear form $\langle \cdot \;, \cdot \rangle _{\tilde\beta}$ as follows. Set $\beta_{00} = -1$, $\beta_{i0}=0=\beta_{0i}$ and $\beta_{ij} = {\tilde\beta}_{ij}$ for $i,j \in \{1, \ldots, m+2n-1\}$. Then the corresponding form $\langle \cdot \;, \cdot \rangle _\beta$ is a supersymmetric, non-degenerate, even bilinear form on the super-vector space $J_\K$ where, for $\K=\mathbb R$, the even part has signature $(m-1,1)$.

Define $(\beta^{ij})_{ij}$ as the inverse of $(\beta_{ij})_{ij}$. Let $\big(e^i\big)_i$ be the right dual basis of $(e_i)_i$ with respect to the form $\langle \cdot \;, \cdot \rangle _\beta$, i.e.,
\begin{gather*}
\langle e_i \;, e^j \rangle _\beta = \delta_{ij}.
\end{gather*}
Then
\begin{gather*}
e^j = \sum_i e_i\beta^{ij}.
\end{gather*}

In this paper we will assume that the orthosymplectic metric is standardized such that
\begin{align*}
\beta = (\beta_{ij})_{i,j = 0}^{m+2n-1} = \left(\begin{array}{cc|cc}
-1 & & &\\
 & I_{m-1} & &\\ \hline
 & & &-I_n\\
 & & I_n &
\end{array}\right),
\end{align*}
with $I_d$ the $d$-dimensional identity matrix.

From now on the real spin factor Jordan superalgebra will always be denoted by $J$ or $J_\R$ and its complexified version by $J_\C$.

\subsection{The TKK algebra}\label{SSExpReal}
With each Jordan (super)algebra one can associate a $3$-graded Lie (super)algebra via the TKK-construction. There exist different TKK-constructions in the literature, see \cite{BC2} for an overview, but for the spin factor Jordan superalgebra $J_\K$ all constructions lead to the orthosymplectic Lie superalgebras $\mf{osp}_\K(m,2|2n)$. We will quickly review the Koecher construction.
First consider $\Inn (J_\K)$, the subalgebra of $\mf{gl}(J_\K)$ of inner derivations. It is generated by the operators $[L_u,L_v]$, $u,v\in J_\K$. If we add the left multiplication operators $L_u$ to the inner derivations we obtain the inner structure algebra:
\begin{gather*}
\mf{istr}(J_\K) := \{L_u|u\in J_\K\}\oplus\Inn(J_\K) = \operatorname{span}_\K\{L_u,[L_u,L_v]\colon u,v\in J_\K\}.
\end{gather*}
Let $J_\K^+$ and $J_\K^-$ be two copies of $J_\K$. As a vector space we define the TKK-algebra of $J_\K$ as
\begin{gather*}
\TKK(J_\K) := J_\K^-\oplus \mf{istr}(J_\K) \oplus J_\K^+.
\end{gather*}
The Lie bracket on $\TKK(J_\K)$ is defined as follows. We interpret $\mf{istr}(J_\K)$ as a subalgebra of~$\mf{gl}(J_\K)$ and for homogeneous $x,y\in J_\K^+$, $u,v\in J_\K^-$, $a,b\in J_\K$ we set
\begin{alignat*}{3}
& [x,u] = 2L_{xu}+2[L_x,L_u], \qquad&& [x,y] = [u,v] = 0,& \\
& [L_a,x] = ax, \qquad && [L_a,u] = -au,& \\
& [[L_a,L_b], x] = [L_a,L_b]x, \qquad && [[L_a,L_b], u] = [L_a,L_b]u.&
\end{alignat*}
From \cite[Proposition~3.1]{BF} we obtain the following
\begin{Theorem}
We have
\begin{gather*}
\mf{istr}(J_\K) = \mf{osp}_\K(J_\K)\oplus \K L_{e_0} \cong \mf{osp}_\K(m-1,1)\oplus \K L_{e_0}, \qquad \TKK(J_\K) \cong \mf{osp}_\K(m,2|2n),
\end{gather*}
where the direct sum decomposition is as algebras.
\end{Theorem}

Using the bilinear form
\begin{gather*}
\overline{\beta} = \left(\begin{array}{cccc|c}
1 & & & &\\
 & (\beta_{ij})_{i,j=1}^{m-1} & & &\\
 & & \beta_{00} & &\\
 & & & -1 &\\ \hline
 & & & & (\beta_{ij})_{i,j = m}^{m+2n-1}
\end{array}\right),
\end{gather*}
we have the following explicit isomorphism of $\TKK(J_\K)$ with the differential operator realisation of $\mathfrak{osp}_\K(m,2|2n)$:
\begin{alignat*}{3}
&e_i^- \mapsto L_{\tilde{i}, (m+1)} +L_{\tilde{i}, 0},\qquad && e_0^- \mapsto L_{m, (m+1)}+L_{m,0},&\\
& L_{e_i} \mapsto L_{\tilde{i}, m},\qquad && L_{e_0} \mapsto L_{0, (m+1)},&\\
& [L_{e_i}, L_{e_j}] \mapsto L_{\tilde{i}, \tilde{j}}, &&& \\
& e_i^+ \mapsto L_{\tilde{i}, (m+1)} -L_{\tilde{i}, 0},\qquad && e_0^+ \mapsto -L_{m,(m+1)}+L_{m,0}.&
\end{alignat*}
Here $\tilde{i} = i$ if $|i|=0$ and $\tilde{i}= i+2$ if $|i|=1$.

\section[The Schr\" odinger model of $\mf{osp}(m,2|2n)$]{The Schr\" odinger model of $\boldsymbol{\mf{osp}(m,2|2n)}$}
\label{The Schrodinger representation}

In the bosonic case (i.e., $n=0$), the Schr\" odinger model of the minimal representation of a~Hermitian Lie group $G$ of tube type can be seen as a representation realized on the Hilbert space~$L^2(C)$ where~$C$ is a minimal orbit for the structure group of~$G$, see~\cite{HKMO}. In the super setting the classical definition of minimal orbits no longer works. Indeed, supermanifolds, in contrast to ordinary manifolds, are not completely determined by their points. In~\cite{BF} an orbit was instead defined as the quotient supermanifold of the structure group by a stabilizer subgroup together with an embedding. Using this definition, a minimal orbit $C$ was constructed which can be characterized by $R^2=0$, with~$R^2$ as introduced in~(\ref{EqR2ELap}). We refer to Section~4 of~\cite{BF} for a detailed description of this minimal orbit. We will now recall the Schr\"odinger model constructed in~\cite{BF}. A critical role in this construction is played by the Bessel operators.

\subsection{The Bessel operator}

The Bessel operator $\bessel_\lambda(x_k)$ is a linear operator acting on $\mc P\big(\R^{m|2n}\big)$. It depends on a complex parameter $\lambda$ and an explicit expression is given by
\begin{gather*}
\bessel_\lambda(x_k) := (-\lambda + 2\E)\pt k - x_k\Delta,
\end{gather*}
where $\E$ and $\Delta$ are the Euler operator and Laplacian introduced in~(\ref{EqR2ELap}).

From Proposition 4.9 of \cite{BF} we obtain the following.
\begin{Proposition}\label{PropTangBes}
The Bessel operators map $\langle R^2 \rangle $ into $\big\langle R^2 \big\rangle $ if and only if $\lambda = 2-M$, where $M=m-2n$ is the superdimension of $\R^{m|2n}$ and $\big\langle R^2 \big\rangle $ is the ideal in $\mc P\big(\R^{m|2n}\big)$ generated by~$R^2$.
\end{Proposition}
Therefore we will only use the Bessel operator with the parameter $2-M$ in this paper and we set $\bessel(x_i) := \bessel_{2-M}(x_i)$. We obtain the following two properties of the Bessel operator from Proposition~4.2 in~\cite{BC1}.

\begin{Proposition}[supercommutativity]
For all $i\in\{0,\ldots, m+2n-1\}$ we have
\begin{gather*}
\bessel(x_i)\bessel(x_j) = (-1)^{|i||j|}\bessel(x_j)\bessel(x_i).
\end{gather*}
\end{Proposition}

\begin{Proposition}[product rule]\label{PropProdRule}
For all $i\in\{0,\ldots, m+2n-1\}$ we have
\begin{gather*}
\bessel(x_i)(\phi\psi) = \bessel(x_i)(\phi)\psi + (-1)^{|i||\phi|}\phi\bessel(x_i)(\psi) \\
\hphantom{\bessel(x_i)(\phi\psi) =}{} + 2(-1)^{|i||\phi|}\E(\phi)\pt i(\psi) +2 \pt i (\phi)\E(\psi) - 2x_i\sum_{r,s}(-1)^{|\phi||r|}\beta^{rs}\pt r(\phi)\pt s(\psi).
\end{gather*}
\end{Proposition}

As a direct result from the product rule we have
\begin{align}\label{EqComBesx}
[\bessel(x_i), x_j] = \beta_{ij}(M-2+2\E)-2L_{ij},
\end{align}
for all $i, j \in\{0, \ldots, m+2n-1\}$.

In what follows we will mostly use the following slightly modified version of the Bessel operator.
\begin{Definition}\label{Modified Bessel Operators}
The \textit{modified Bessel operator} $\mbessel_\lambda(x_k)$ is given by
\begin{gather*}
\mbessel(x_0) :=-\bessel(x_0), \qquad \mbessel(x_k) :=\bessel(x_k),
\end{gather*}
for $k\in \{1, \ldots , m+2n-1\}$.
\end{Definition}

\subsection{Radial superfunctions}\label{SSRadSub}

Let us denote the supervariables $x_m, \ldots, x_{m+2n-1}$ by $\theta_1, \ldots, \theta_{2n}$. We keep the notations for $x_0, \ldots , x_{m-1}$. Define $\theta^I$ as $\theta_1^{i_1}\theta_2^{i_2} \cdots \theta_n^{i_{2n}}$ for $I = (i_1, i_2, \dots, i_{2n}) \in \mathbb{Z}_2^{2n}$. Consider a function $h\colon \R\rightarrow \R$, $h\in \mc C^{2n}(\R\setminus \{0\})$ and a superfunction $f=f_0+\sum\limits_{I\neq 0}f_I\theta^I$, with $f_0, f_I\in \mc C^{\infty}\big(\R^m\big)$ and where $f_0$ has non-zero values. Then a new superfunction
\begin{gather}\label{EqDefRadSup2}
h(f) := \sum_{j=0}^{2n}\dfrac{1}{j!}\left(\sum_{I\neq 0}f_I\theta^I\right)^jh^{(j)}(f_0)
\end{gather}
is defined in \cite[Definition 3]{CDS}. Here $h^{(j)}$ denotes the $j$th derivative of $h$.

Set the supervariable
\begin{gather*}%\label{Eqs2th2}
r^2 := \sum_{i=1}^{m+2n-1}x^i x_i,
\end{gather*}
so $R^2 = -x_0^2+ r^2$. We can use equation~\eqref{EqDefRadSup2} with $f= \big(x_0^2+r^2\big)/2$ and $h$ the square root to define the superfunction $|X|$ as
\begin{gather*}
|X| := \sqrt{\dfrac{x_0^2 + r^2}{2}}.
\end{gather*}
Using equation~\eqref{EqDefRadSup2} again, but now with $f = |X|$ we define $\exp(|X|)$ and $\Lambda_{2,j}^{\mu,\nu}(|X|)$ as radial superfunctions. Here $\Lambda_{2,j}^{\mu,\nu}$ is the generalised Laguerre function introduced in Appendix \ref{AppLag}.

\subsection[The $(\mf g, \mf k)$-module $W$]{The $\boldsymbol{(\mf g, \mf k)}$-module $\boldsymbol{W}$}\label{SSSchrodRep}

We introduce the notations
\begin{gather*}
\mf g := \TKK(J) \cong \mf{osp}(m,2|2n),\\
\mf k := \{(x, I, -x) \colon I \in \Inn(J),\, x\in J\},\\
\mf k_0 := \mf k \cap \mf{istr}(J) = \Inn(J).
\end{gather*}

\begin{Theorem}The following isomorphisms
\begin{gather*}
\mf k_0 \cong \mf{osp}(m-1,0|2n), \qquad
\mf k \cong \mf{osp}(m,0|2n)\oplus \R,
\end{gather*}
hold as algebras.
\end{Theorem}

\begin{proof}This follows from a straightforward verification by, for example, looking at the matrix realisation of $\mf g$.
\end{proof}

In Section~5.2 of~\cite{BF} the so-called Schr\"odinger model of $\mf g$ was constructed for $M\geq 4$. Since this model will play an important role in this paper, we will give a short recall of its construction.
Firstly, we consider a representation $\pi$ of $\mf g$ acting on smooth superfunctions on~$J^-$. Explicitly, $\pi$ of $\mf g = \TKK(J)=J^-\oplus \mf{istr}(J) \oplus J^+$ is given as follows:
\begin{itemize}\itemsep=0pt
\item $\pi (e_l, 0, 0) = -2\imath x_l $,
\item $\pi (0, L_{e_k}, 0) = -x_0\pt k + x_k\pt 0$,
\item $\pi (0, [L_{e_i}, L_{e_j}], 0) = x_i\pt j - (-1)^{|i||j|}x_j\pt i$,
\item $\pi (0, L_{e_0}, 0) = \frac{2-M}{2}-\E$,
\item $\pi (0, 0, e_l) = -\frac{1}{2}\imath \mbessel(x_l)$,
\end{itemize}
with $i,j,k\in\{1, 2, \ldots, m+2n-1\}$, $l\in\{0, 1, \ldots, m+2n-1\}$ and where $\imath$ is the imaginary unit. For $n=0$, our convention corresponds to the Schr\"odinger realisation given in~\cite{HKMO}. It only differs from the Schr\"odinger realisation given in~\cite{BF} by a change of basis. It can be shown that all the operators occuring in~$\pi$ are tangential to~$R^2$, with $R^2$ as in equation~\eqref{EqR2ELap}. So we can consider the representation obtained by quotienting out $R^2$. This quotient representation has an interesting subrepresentation consisting of $\mf k$-finite vectors. It is generated by $\exp(-2|X|)$, with $\exp(-2|X|)$ the smooth radial superfunction on~$J^-$ as defined in Section~\ref{SSRadSub}. This subrepresentation wil be our Schr\"odinger model. So the Schr\"odinger model is given as follows:
\begin{gather*}
W := U(\mf g) \exp(-2|X|) \mod \big\langle R^2\big\rangle ,
\end{gather*}
where the $\mf g$-module structure is given by the \textit{Schr\"odinger representation}~$\pi$.

Theorem~5.3 of~\cite{BF} gives the following decomposition of $W$.

\begin{Theorem}[decomposition of $W$]\label{ThDecW}
Assume $M\geq 4$.
\begin{itemize}\itemsep=0pt
\item[$(1)$] The decomposition of $W$ as a $\mf k$-module is given by
\begin{gather*}
W =\bigoplus_{j=0}^{\infty}W_j, \qquad \mbox{with}\quad W_j =U(\mf k)\Lambda_{2,j}^{M-3,-1}(2|X|),
\end{gather*}
where $W_j$ and thus also $W$ are $\mf k$-finite.
\item[$(2)$] $W$ is a simple $\mf g$-module.
\item[$(3)$] An explicit decomposition of $W_j$ into irreducible $\mf k_0$-modules is given by
\begin{gather*}
W_j = \bigoplus_{k=0}^j \bigoplus_{l=0}^{1}\Lambda_{2,j-k}^{M-3 + 2k, 2l-1}(2|X|)\big(\mc H_k\big(\R^{m-1|2n}\big)\otimes \mc H_l(\R)\big).
\end{gather*}
Furthermore, if $m$ is even we also have the following $\mf k$-isomorphism
\begin{gather*}
W_j \cong \mc H_j\big(\R^{m|2n}\big)\otimes \mc H_{\frac{M-2}{2}+j}\big(\R^{2}\big).
\end{gather*}
\end{itemize}
\end{Theorem}

\subsection{The integral and sesquilinear form}\label{SSIntC}

In Section~8 of~\cite{BF} an integral which restricts to $W$ was constructed. We will use a renormalized version of that integral, restricted to $x_0>0$. To give the integral explicitly we consider spherical coordinates in $\R^{m-1}$ by setting $x_i = \omega_i s$ with $\omega_i \in \mathbb S^{m-2}$, $i\in\{1, \ldots, m-1\}$ and
\begin{gather*}
s^2 := \sum_{i=1}^{m-1}x_i^2.
\end{gather*}
We also introduce the following notations
\begin{gather*}
\theta^2 := \sum_{i,j=m}^{m+2n-1}\beta^{ij}x_ix_j, \qquad 1+\eta := \sqrt{1-\dfrac{\theta^2}{2s^2}}, \qquad 1+\xi := \sqrt{1+\dfrac{\theta^2}{2x_0^2}}
\end{gather*}
and the morphism
\begin{gather*}
\phi^\sharp (f) := \sum_{j=0}^n \dfrac{\theta^{2j}}{j!}\left(\dfrac{1}{4x_0}\pt {x_0} - \dfrac{1}{4s}\pt s\right)^j(f),
\end{gather*}
for $f\in \mc C^{\infty}\big(\R^m\setminus \{0\}\big)\otimes \Lambda\big(\R^{2n}\big)$. We remark that in \cite[Lemma~8.2]{BF} it is shown that $\phi^\sharp$ is actually an algebra isomorphism. The Berezin integral on $\Lambda\big(\R^{2n}\big)$ is defined as
\begin{gather*}
\int_B := \partial^{m+2n-1}\partial^{m+2n-2}\cdots \partial^{m}.
\end{gather*}

\begin{Definition}\label{DefIntW}
Suppose $M\geq 4$. For $f\in W$ we define the integral $\int_W$ by
\begin{gather}\label{EqExplW}
\int_W f := \dfrac{1}{\gamma} \int_0^\infty\int_{\mathbb S^{m-2}}\int_B\rho^{m-3}(1+\eta)^{m-3}(1+\xi)^{-1}\phi^\sharp(f)_{|s=x_0=\rho}\mathrm{d}\rho \mathrm{d}\omega,
\end{gather}
where $\gamma\in \C$ is the renormalisation factor such that $\int_W \exp(-4|X|)=1$.
\end{Definition}

We can show that the integral is well defined modulo $\big\langle R^2\big\rangle $. This follows from Section~8 of~\cite{BF} together with the fact that~$\gamma$ is non-zero.

\begin{Proposition}
For $M = m-2n\geq 4$ we have
\begin{gather*}
\gamma = \dfrac{2^{5-2M}}{n!} \left(\dfrac{3-m}{2}\right)_n \dfrac{\pi^{\frac{m-1}{2}}}{\Gamma{\big(\frac{m-1}{2}\big)}}\Gamma(M-2),
\end{gather*}
where we used the Pochhammer symbol $(a)_k = a(a+1)(a+2)\cdots (a+k-1)$. Note that $\big(\frac{3-m}{2}\big)_n=0$ implies $M = m-2n\leq 1$. Therefore~$\gamma$ is non-zero.
\end{Proposition}

\begin{proof}Let us denote the integral $\int_W$ not normalized by $\gamma$ as $\int_{W'}$, i.e., $\int_{W'} = \gamma \int_W$. We wish to calculate
\begin{gather*}%\label{Eqgamma}
\gamma = \int_{W'}\exp(-4|X|).
\end{gather*}
In \cite[Lemma 6.3.5]{PhDB} a similar integral is calculated if one observes that
$\widetilde K_{-\frac{1}{2}}(t) = \frac{\sqrt{\pi}}{2}\exp(-t)$, (equation~\eqref{Eq Bessel is exp} in Appendix~\ref{AppBes})
where $\widetilde K_\alpha$ is the K-Bessel function introduced in Appendix~\ref{AppBes}.
Using the same calculations as in the proof of \cite[Lemma~6.3.5]{PhDB}, we then obtain
\begin{gather*}
\gamma = \dfrac{1}{n!} \left(\dfrac{3-m}{2}\right)_n \dfrac{\pi^{\frac{m-3}{2}}}{\Gamma{\big(\frac{m-1}{2}\big)}}\dfrac{\Gamma\big(\frac{M}{2}\big) \Gamma\big(\frac{M}{2}-1\big)\Gamma\big(\frac{M-1}{2}\big)^2}{\Gamma(M-1)},
\end{gather*}
provided we take into account that we need extra factors $2$ in certain places and that we restrict ourselves to $x_0>0$. Note that in \cite[Lemma~6.3.5]{PhDB} the tilde in $\widetilde K_{-\frac{1}{2}}$ sometimes mistakenly disappears. If we use Legendre's duplication formula:
\begin{gather*}
\Gamma\left(\dfrac{z+1}{2}\right)\Gamma\left(\dfrac{z}{2}\right) = 2^{1-z}\sqrt{\pi}\Gamma(z),
\end{gather*}
on $\Gamma\big(\frac{M}{2}\big)\Gamma\big(\frac{M-1}{2}\big)$ and $\Gamma\big(\frac{M-1}{2}\big)\Gamma\big(\frac{M}{2}-1\big)$ the result follows.
\end{proof}

\begin{Definition}For $f,g\in W$ we define the sesquilinear form $\ip{\cdot\, ,\cdot}$ as
\begin{gather*}
\ip{f,g} := \int_W f\overline{g}.
\end{gather*}
\end{Definition}

Theorem~8.13 and Lemma~8.14 in \cite{BF} give us the following two properties.

\begin{Proposition}\label{PropSkewSymPi}
Suppose $M = m-2n \geq 4$. The Schr\" odinger representation $\pi$ on $W$ is skew-supersymmetric with respect to $\ip{\cdot\, ,\cdot}$, i.e.,
\begin{gather*}
\ip{\pi(X)f,g} = - (-1)^{|X||f|}\ip{f,\pi(X)g},
\end{gather*}
for all $X\in \mathfrak{g}$ and $f,g\in W$.
\end{Proposition}

\begin{Proposition}
Suppose $M = m-2n \geq 4$. The form $\ip{\cdot\, ,\cdot}$ defines a sesquilinear, non-degenerate form on $W$, which is superhermitian, i.e.,
\begin{gather*}
\ip{f,g} = (-1)^{|f||g|}\ipbar{g,f},
\end{gather*}
for all $f,g\in W$.
\end{Proposition}

Note that for both Theorem~8.13 and Lemma~8.14 in \cite{BF} there is an extra condition saying that $M$ must be even. However, since we are working in the exceptional case that corresponds with $q=2$ in \cite{BF}, the proofs still hold without this extra condition.

\section{The Fock space}\label{The Fock space}
In \cite[Section 2.3]{HKMO} an inner product on the polynomial space $\mc P(\C^m)$ was introduced, namely the Bessel--Fischer inner product
\begin{gather*}
\bfip{p,q} := p\big(\mbessel\big)\bar q(z)\big|_{z=0},
\end{gather*}
where $\bar q(z) = \overline{q(\bar z)}$ is obtained by conjugating the coefficients of the polynomial $q$. In \cite[Pro\-po\-si\-tion~2.6]{HKMO} it is proven that, for polynomials, the Bessel--Fischer inner product is equal to the $L^2$-inner product of the Fock space. Since there is no immediate extension of this $L^2$-inner product to the super setting, we will use the Bessel--Fischer inner product on polynomials as the starting point to generalize the Fock space to superspace.

\subsection{Definition and properties}

\begin{Definition}\label{DefFock}
We define the \textit{polynomial Fock space} as the superspace
\begin{gather*}
\Fock := \mathcal{P}\big(\C^{m|2n}\big)/\big\langle R^2\big\rangle.
\end{gather*}
\end{Definition}
From the generalised Fischer decomposition of Proposition \ref{PropFDgen} it follows that
\begin{gather*}
\Fock \cong\bigoplus_{l=0}^\infty \gsh_l\big(\C^{m|2n}\big).
\end{gather*}
In particular, if the superdimension $M=m-2n$ is such that $M\not\in -2\N$, then the Fischer decomposition from Proposition \ref{LemFC} gives us
\begin{align*}
\Fock \cong\bigoplus_{l=0}^\infty \mathcal{H}_l\big(\C^{m|2n}\big).
\end{align*}

\begin{Definition}For $p, q\in \mathcal P\big(\C^{m|2n}\big)$ we define the \textit{Bessel--Fischer product} of $p$ and $q$ as
\begin{gather*}
\bfip{p,q} := p\big(\mbessel\big)\bar q(z)\big|_{z=0},
\end{gather*}
where $\bar q(z) = \overline{q(\bar z)}$ is obtained by conjugating the coefficients of the polynomial~$q$ and~$\mbessel$ is the complex version of the modified Bessel operators introduced in Definition~\ref{Modified Bessel Operators}.
\end{Definition}

Explicitly for $p=\sum_\alpha a_\alpha z^\alpha$ and $q=\sum_\beta b_\beta z^\beta$ we have
\begin{gather*}
\bfip{p,q} = \sum_{\alpha, \beta}a_\alpha \bar b_\beta (-1)^{\alpha_0}\bessel(z_0)^{\alpha_0}\cdots\bessel(z_{m+2n-1})^{\alpha_{m+2n-1}}z_0^{\beta_0}\cdots z_{m+2n-1}^{\beta_{m+2n-1}}\big|_{z=0}.
\end{gather*}
Note that it is only an inner product in the bosonic case. However, in \cite{dGM} a new definition of Hilbert superspaces was introduced where the preserved form is no longer an inner product, but rather a non-degenerate, sesquilinear, superhermitian form. We will prove that the Bessel--Fischer product is such a form when restricted to $\Fock$ with $M-2\not\in -2\N$.

\begin{Proposition}[sesquilinearity]
For $p, q, r, s\in \mathcal P\big(\C^{m|2n}\big)$ and $\alpha, \beta, \gamma, \delta\in\C$ we have
\begin{gather*}
\bfip{\alpha p+ \gamma r, \beta q + \delta s} = \alpha\bar\beta \bfip{p, q} + \alpha\bar\delta \bfip{p, s} + \gamma\bar\beta \bfip{r, q} + \gamma\bar\delta \bfip{r, s}.
\end{gather*}
\end{Proposition}
\begin{proof}
This follows from the linearity of the Bessel operators.
\end{proof}

\begin{Proposition}[orthogonality]\label{PropOrthog}
For $p_k\in \mathcal P_k\big(\C^{m|2n}\big)$ and $p_l\in \mathcal P_l\big(\C^{m|2n}\big)$ with $l\neq k$ we have
$\bfip{p_k, p_l} = 0$.
\end{Proposition}
\begin{proof}
This follows from the fact that Bessel operators lower the degree of polynomials by one.
\end{proof}

\begin{Proposition}\label{degreeshift}
For $p, q\in \mathcal P\big(\C^{m|2n}\big)$ and $i\in\{0, 1, \ldots , m+2n-1\}$ we have
\begin{gather*}
\bfip{z_i p, q} = (-1)^{|i||p|}\big\langle p,\mbessel(z_i)q\big\rangle_{\mathcal B}.
\end{gather*}
\end{Proposition}
\begin{proof}
This follows immediately from the definition of the Bessel--Fischer product.
\end{proof}

To prove that the Bessel--Fischer product is superhermitian we will use induction on the degree of the polynomials. We will need the following lemma in the induction step of the proof.

\begin{Lemma}\label{LemLij}
Suppose we have proven that $\bfip{p, q} = (-1)^{|p||q|}\bfipbar{q, p}$
for all $p,q\in \mc P\big(\C^{m|2n}\big)$ of degree lower than or equal to $k\in \N$. Then for $p,q\in \mc P_l\big(\C^{m|2n}\big)$, $l\leq k$ we have
\begin{gather*}
\bfip{ L_{ij} p, q} = -(-1)^{(|i|+|j|)|p|}\bfip{ p, L_{ij} q}
\qquad \text{and} \qquad
\bfip{ L_{0i} p, q} = (-1)^{|i||p|}\bfip{ p, L_{0i} q},
\end{gather*}
for all $i,j\in\{1,\ldots, m+2n-1\}$.
\end{Lemma}

\begin{proof}We will prove the $L_{ij}$ case explicitly. The $L_{0i}$ case is entirely analogous. First note that if we combine the given superhermitian property with Proposition~\ref{degreeshift}, we get
\begin{gather*}
\big\langle\mbessel(z_i) p, q\big\rangle_{\mathcal B} = (-1)^{|i||p|}\bfip{p,z_i q}.
\end{gather*}
for all $p\in \mc{ P}\big(\C^{m|2n}\big)$ of degree $k$ or lower and all $q\in \mc{P} \big(\C^{m|2n}\big)$ of degree $k-1$ or lower. Assume $p, q\in \mc P_l\big(\C^{m|2n}\big)$, $l\leq k$. We obtain
\begin{gather*}
\bfip{z_i \pt j p,q} = (-1)^{|i||p|}\bfip{ \pt j p, \mbessel(z_i) q} \\
\hphantom{\bfip{z_i \pt j p,q}}{} = (-1)^{|i|(|p|+|j|)} \big((M-2+2(l-1))\bfip{ \pt j p, \pt i q} - \bfip{ \pt j p, z_i\Delta q}\big)\\
\hphantom{\bfip{z_i \pt j p,q}}{}= (-1)^{|i|(|p|+|j|)}\big( \bfip{ \mbessel(z_j) p, \pt i q} + \bfip{ z_j\Delta p, \pt i q} - \bfip{ \pt j p, z_i\Delta q}\big)\\
\hphantom{\bfip{z_i \pt j p,q}}{}= (-1)^{(|i|+|j|)|p| +|i||j|} \bfip{ p, z_j\pt i q} +(-1)^{|i|(|p|+|j|)} \big( \bfip{ z_j\Delta p, \pt i q} - \bfip{ \pt j p, z_i\Delta q}\big).
\end{gather*}
This gives
\begin{gather*}
\bfip{L_{ij}p,q} =\bfip{z_i\pt j p,q}-(-1)^{|i||j|}\bfip{z_j \pt i p,q}\\
\hphantom{\bfip{L_{ij}p,q}}{} = (-1)^{(|i|+|j|)|p| +|i||j|} \bfip{ p, z_j\pt i q} - (-1)^{(|i|+|j|)|p|} \bfip{p, z_i\pt j q} \\
\hphantom{\bfip{L_{ij}p,q}=}{} +(-1)^{|i|(|p|+|j|)} \big( \bfip{ z_j\Delta p, \pt i q} - \bfip{ \pt j p, z_i\Delta q}\big)\\
\hphantom{\bfip{L_{ij}p,q}=}{} -(-1)^{|j||p|} \big( \bfip{ z_i\Delta p, \pt j q} - \bfip{ \pt i p, z_j\Delta q}\big)\\
\hphantom{\bfip{L_{ij}p,q}}{} = -(-1)^{(|i|+|j|)|p|}\bfip{ p, L_{ij} q}
 +(-1)^{|i|(|p|+|j|)} \big( \bfip{ z_j\Delta p, \pt i q} - \bfip{ \pt j p, z_i\Delta q}\big)\\
\hphantom{\bfip{L_{ij}p,q}=}{} -(-1)^{|j||p|} \big( \bfip{ z_i\Delta p, \pt j q} - \bfip{ \pt i p, z_j\Delta q}\big),
\end{gather*}
from which the desired result follows if we prove
\begin{gather}
0 = (-1)^{|i|(|p|+|j|)} \bfip{ z_j\Delta p, \pt i q} - (-1)^{|i|(|p|+|j|)} \bfip{ \pt j p, z_i\Delta q}\nonumber\\
\hphantom{0=}{} -(-1)^{|j||p|} \bfip{ z_i\Delta p, \pt j q} + (-1)^{|j||p|}\bfip{ \pt i p, z_j\Delta q}.\label{Eq0Lij}
\end{gather}
For the first term in right hand side of this equation we have
\begin{gather*}
\bfip{ z_j\Delta p, \pt i q} = (-1)^{|p||j|}\big\langle \Delta p, \mbessel(z_j)\pt i q\big\rangle_{\mathcal B}\\
\hphantom{\bfip{ z_j\Delta p, \pt i q}}{} = (-1)^{|p||j|}\big((M-2+2(l-1))\bfip{\Delta p, \pt j\pt i q}- \bfip{ \Delta p, z_j\pt i\Delta q}\big),
\end{gather*}
such that, using similar calculations for the other three terms, equation~\eqref{Eq0Lij} can be rewritten as
\begin{gather*}
\bfip{L_{ij}\Delta p, \Delta q} = -(-1)^{(|i|+|j|)|p|}\bfip{\Delta p, L_{ij}\Delta q}.
\end{gather*}
Since $\Delta p$ and $\Delta q$ are polynomials of degree lower than $l$ the lemma follows from a straightforward induction argument on $l$.
\end{proof}

\begin{Proposition}[superhermitianity]\label{supersym}The Bessel--Fischer product is superhermitian, i.e.,
\begin{gather*}
\bfip{p, q} = (-1)^{|p||q|}\bfipbar{q, p},
\end{gather*}
for $p, q\in \mc P\big(\C^{m|2n}\big)$.
\end{Proposition}

\begin{proof}Because of the orthogonality we only need to prove the property for $p$, $q$ homogeneous polynomials of the same degree. Because of the sesquilinearity we may assume that $p$ and $q$ are monomials. We will use induction on the degree $k$ of the polynomials, the case $k=0$ being trivial. Suppose we have proven the theorem for all $p, q \in \mathcal P_k\big(\C^{m|2n}\big)$. We now look at $\bfip{z_i p , z_j q}$ for arbitrary $i,j\in\{0,1, \ldots, m+2n-1\}$. To simplify the following calculations we will restrict ourselves to $i,j\neq 0$, the cases $i=0$ or $j=0$ being similar. Denote $c = (M-2+2k)$. Using the commutation of equation~\eqref{EqComBesx} and Proposition~\ref{degreeshift} we find
\begin{align*}
\bfip{z_i p , z_j q} &= (-1)^{|i||p|} \bfip{p , \bessel(z_i)z_j q}\\
&= (-1)^{|i||p|+|i||j|} \bfip{p , z_j \bessel(z_i) q} + (-1)^{|i||p|} \bfip{p , [\bessel(z_i),z_j] q}\\
&= (-1)^{|i||p|+|i||j|} \bfip{p , z_j \bessel(z_i) q} + (-1)^{|i||p|}c\beta_{ij}\bfip{ p, q} -2(-1)^{|i||p|}\bfip{ p, L_{ij} q}.
\end{align*}
Using the induction hypothesis together with Proposition \ref{degreeshift} this becomes
\begin{gather*}
\bfip{z_i p , z_j q} = (-1)^{(|i|+|j|)|p|+|i||j|} \bfip{\bessel(z_j) p , \bessel(z_i) q} + (-1)^{|i||p|}c\beta_{ij}\bfip{ p, q}\\
\hphantom{\bfip{z_i p , z_j q} =}{} -2(-1)^{|i||p|}\bfip{ p, L_{ij} q}.
\end{gather*}
Switching the roles of $z_i p$ and $z_j q$ we also obtain
\begin{gather*}
\bfip{ z_j q, z_i p} = (-1)^{(|i|+|j|)|q|+|i||j|} \bfip{\bessel(z_i) q , \bessel(z_j) p} + (-1)^{|j||q|}c\beta_{ji}\bfip{q, p} \\
\hphantom{\bfip{ z_j q, z_i p}=}{}
 -2(-1)^{|j||q|}\bfip{q, L_{ji} p}\\
\hphantom{\bfip{ z_j q, z_i p}}{}
= (-1)^{|i||q|+|i||p|+|q||p|} \bfip{\bessel(z_j) p , \bessel(z_i) q} + (-1)^{|j||q|+|i||j|+|p||q|}c\beta_{ij}\bfip{p,q}\\
\hphantom{\bfip{ z_j q, z_i p}=}{} +2(-1)^{|i||q|+|q||p|+|i||j|}\bfip{L_{ij} p, q},
\end{gather*}
where we used the induction hypothesis on all three terms of the right hand side. If we use Lemma \ref{LemLij} on the last term and multiply both sides of this equation with $(-1)^{(|p|+|i|)(|q|+|j|)}$ we get
\begin{gather*}
(-1)^{(|p|+|i|)(|q|+|j|)}\bfip{ z_j q, z_i p} = (-1)^{|i||j|+|i||p|+|j||p|} \bfip{\bessel(z_j) p , \bessel(z_i) q}\\
\hphantom{(-1)^{(|p|+|i|)(|q|+|j|)}\bfip{ z_j q, z_i p} =}{} + (-1)^{|j||q|+|i||q|+|p||j|}c\beta_{ij}\bfip{p,q} - 2(-1)^{|i||p|}\bfip{ p, L_{ij}q}\\
\hphantom{(-1)^{(|p|+|i|)(|q|+|j|)}\bfip{ z_j q, z_i p}}{}= \bfip{z_i p , z_j q},
\end{gather*}
where we made use of $\beta_{ij} = 0$ for $|i|\neq |j|$ in the last step.
\end{proof}

Lemma \ref{LemLij} can now be extended to the following proposition.

\begin{Corollary}\label{PropLij}We have
\begin{gather*}
\bfip{ L_{ij} p, q} = -(-1)^{(|i|+|j|)|p|}\bfip{ p, L_{ij} q}
\qquad \text{and}\qquad
\bfip{ L_{0i} p, q} = (-1)^{|i||p|}\bfip{ p, L_{0i} q},
\end{gather*}
for all $i,j\in\{1,\ldots, m+2n-1\}$ and $p,q\in \mc P\big(\C^{m|2n}\big)$.
\end{Corollary}
 We have the following proposition.
\begin{Proposition}\label{Restriction}
For all $p,q\in \mc P\big(\C^{m|2n}\big)$ it holds that
\begin{align*}
\big\langle R^2p,q\big\rangle_{\mathcal B} =0 = \big\langle p,R^2q\big\rangle_{\mathcal B}.
\end{align*}
Thus the Bessel--Fischer product is well-defined on $\Fock$.
\end{Proposition}

\begin{proof}
Because of the superhermitian property we only need to prove $\big\langle p,R^2q\big\rangle = 0$ holds for all $p,q\in \mc P\big(\C^{m|2n}\big)$. The complex version of Proposition \ref{PropTangBes} implies that for arbitrary $p,q\in \mathcal P\big(\C^{m|2n}\big)$ there exists a $q'\in \mathcal P\big(\C^{m|2n}\big)$ such that
\begin{gather*}
\big\langle p, R^2 q\big\rangle_{\mathcal B} = p\big(\mbessel\big)R^2q(z) |_{z=0} = R^2 q'(z) |_{z=0}.
\end{gather*}
Since the constant term of $R^2 q'(z)$ is always zero this proves the theorem.
\end{proof}

Note that the above proposition also shows that the Bessel--Fischer product is degenerate on the space $\mathcal P\big(\C^{m|2n}\big)$. In the following subsection we look at the non-degeneracy of the Bessel--Fischer product when we restrict it to $\Fock$.

\subsection{The reproducing kernel and non-degeneracy}

In the bosonic case a reproducing kernel for the Fock space was constructed in Section~2.4 of~\cite{HKMO}. We will prove the non-degeneracy of the Bessel--Fischer product on~$\Fock$ by first constructing a~generalisation of this reproducing kernel in superspace.

\begin{Lemma}\label{Lemma rep kernel}
Suppose $M-2\not\in -2\N$ and $k\in \N$. Define the superfunction $\mathbb K^k(z,w)$ by
\begin{gather*}
\mathbb K^k(z,w) := \dfrac{1}{4^k k!}\left(\dfrac{M}{2}-1\right)_k^{-1}(z|\overline{w})^k,
\end{gather*}
where we used the Pochhammer symbol $(a)_k = a(a+1)(a+2)\cdots (a+k-1)$ and $z|w$ is defined as
\begin{gather*}
z|w := 2z_0w_0+2\sum_{i,j=1}^{m+2n-1}z_i\beta^{ij}w_j.
\end{gather*}
For all $p\in \mc P_k\big(\C^{m|2n}\big)$ we then have
\begin{align*}
\big\langle p,\mathbb K^k(z,w)\big\rangle_{\mathcal B} = p(w) \mod \big\langle R^2_w\big\rangle .
\end{align*}
\end{Lemma}

\begin{proof}We have
\begin{align*}
\mbessel(z_l)(z|w)^k &= (-1)^{\delta_{0l}}\left((M-2+2\E)\pt l - z_l\Delta\right)(z|w)^k\\
&= 2w_l k(M+2k-4)(z|w)^{k-1} - 4z_l k (k-1)R^2_w(z|w)^{k-2}\\
&= 4 w_l k \left(k+\dfrac{M}{2}-2\right)(z|w)^{k-1}-R^2_w 4z_l k (k-1)(z|w)^{k-2}.
\end{align*}
Now suppose $p$ is a monomial, i.e.,
$p(z) = a\prod\limits_{l=0}^{m+2n-1}z_l^{\alpha_l}$,
with $|\alpha|=k$ and $a\in\C$. Iterating the previous calculation and working modulo $R^2_w$ we obtain
\begin{align*}
\bfip{p,(z|\overline{w})^k} &= \left. p\big(\mbessel\big)(z|w)^k\right|_{z=0}
 = a\prod_{l=0}^{m+2n-1}\left.\mbessel(z_l)^{\alpha_l}(z|w)^k\right|_{z=0}\\
&= 4^k (k(k-1)\cdots 1)\left(k+\dfrac{M}{2}-2\right)\cdots \left(\dfrac{M}{2}-1\right)a \prod_{l=0}^{m+2n-1} w_l^{\alpha_l}\\
&= 4^k k! \left(\dfrac{M}{2}-1\right)_k p(w),
\end{align*}
which gives us the desired result.
\end{proof}

\begin{Theorem}[reproducing kernel of $\Fock$]\label{Theorem repr kernel}
Suppose $M-2\not\in -2\N$ and define the superfunction $\mathbb K(z,w)$ by
\begin{gather*}
\mathbb K(z,w) := \Gamma\left(\dfrac{M}{2}-1\right)\widetilde I_{\frac{M}{2}-2}\big(\sqrt{(z|\overline w)}\big) = \sum_{k=0}^\infty \dfrac{1}{4^k k!}\left(\dfrac{M}{2}-1\right)_k^{-1}(z|\overline{w})^k,
\end{gather*}
where $\tilde{I}_{\alpha}$ is the I-Bessel introduced in Appendix~{\rm \ref{AppBes}}. For all $p\in \Fock$ we have
\begin{gather*}
\bfip{p,\mathbb K(z,w)} = p(w).
\end{gather*}
\end{Theorem}

\begin{proof}By Lemma \ref{Lemma rep kernel}, the orthogonality property and the Fischer decomposition we find that
\begin{gather*}
\sum_{k=0}^\infty \mathbb K^k(z,w) = \sum_{k=0}^\infty \dfrac{1}{4^k k!}\left(\dfrac{M}{2}-1\right)_k^{-1}(z|\overline{w})^k,
\end{gather*}
has the desired property.
\end{proof}

The non-degeneracy of the Bessel--Fischer product on $\Fock$ is now an almost immediate result.

\begin{Proposition}[non-degenerate case]\label{nondeg}
For $M-2\not\in -2\N$ the Bessel--Fischer product is non-degenerate on the polynomial Fock space~$\Fock$, i.e., if
$
\bfip{p,q} =0,
$
for all $q\in \Fock$, then $p=0$.
\end{Proposition}

\begin{proof}Suppose $p \in \Fock$ is such that $\bfip{p,q} = 0$ for all $q\in \Fock$. Using the reproducing kernel we obtain $p(w) = \bfip{p,\mathbb K(z,w)} =0$. Hence $p=0$.
\end{proof}

Note that the previous proposition only works when $M-2\not\in -2\N$. For the $M-2\in {-}2\N$ case the Bessel--Fischer product will always be degenerate.

\begin{Proposition}[degenerate case]
For $M-2\in -2\N$ the Bessel--Fischer product is degenerate on the polynomial Fock space $\Fock$, i.e., there exists a $q\in \Fock$, $q\neq 0$ such that
$
\bfip{p,q} =0,
$
for all $p\in \Fock$.
\end{Proposition}
\begin{proof}Suppose $M-2\in -2\N$ and $q\in \mc H_{2-\frac{M}{2}}\big(\C^{m|2n}\big)\subset\gsh_{2-\frac{M}{2}}\big(\C^{m|2n}\big)\subset \Fock$. Then
\begin{gather*}
\mbessel(z_i)q = \pm ((M-2 +2\E)\pt i q - z_i\Delta q) = 0,
\end{gather*}
for all $i\in\{0, 1, \ldots, m+2n-1\}$. For all $p\in \Fock$ it follows then immediately from the definition that
$\bfip{p,q} =0$. Since
\begin{gather*}
\dim \mc H_{2-\frac{M}{2}}\big(\C^{m|2n}\big) \geq \dim \mc P_{2-\frac{M}{2}}\big(\C^{m-1|2n}\big) \geq 1,
\end{gather*}
we conclude that such a $q\neq 0$ exists.
\end{proof}

\section[The Fock model of $\mf{osp}(m,2|2n)$]{The Fock model of $\boldsymbol{\mf{osp}(m,2|2n)}$}\label{The Fock model}
In \cite{HKMO} the Fock representation $\rho := \pi_\C \circ c$ is obtained by twisting the complexification of the Schr\" odinger representation $\pi_\C$ with the Cayley transform~$c$.
 This Cayley transform is an isomorphism of $\mf{g}_\C$ which induces a Lie algebra isomorphism between~$\mf k_\C$ and $\mathfrak{istr}(J_\C)$. We will use a similar approach in our construction of the Fock model. We start by defining a Cayley transform~$c$ in our setting.
\subsection{Definition and properties}
Let $\mf{g}_\C$ and $\mf{k}_\C$ be the complexified versions of $\mf g$ and $\mf k$ introduced in Section~\ref{SSSchrodRep}, i.e.,
\begin{gather*}
\mathfrak{g}_\C = \TKK(J_\C) \cong \mathfrak{osp}_\C(m+2|2n),\\
\mf k_\C = \{(z, I, -z) \colon I \in \Inn(J_\C), z\in J_\C\}\cong \mf{osp}_\C(m|2n)\oplus \C,
\end{gather*}
Let $\imath$ denote the complex unit. Define the Cayley transform $c \in \End(\mf{g}_\C)$ as
\begin{gather*}
c := \exp\left(\frac{\imath}{2}\ad(e_0^-)\right)\exp\big(\imath\ad\big(e_0^+\big)\big),
\end{gather*}
with $e_0^-$ and $e_0^+$ the copies of $e_0$ in $J^-$ and $J^+$, respectively.
\begin{Proposition} Using the decomposition $\mf{g}_\C= J^-_\C\oplus \mathfrak{istr}(J_\C)\oplus J^+_\C $ we obtain the following explicit expression for the Cayley transform
\begin{itemize}\itemsep=0pt
\item $c(a,0,0) = \left(\dfrac{a}{4}, \imath L_a, a\right)$,
\item $c(0,L_a+I,0) = \left(\imath \dfrac{a}{4}, I, -\imath a\right)$,
\item $c(0,0,a) = \left(\dfrac{a}{4}, -\imath L_a, a\right)$,
\end{itemize}
with $a\in J$ and $I\in \Inn(J_\C)$. It induces a Lie superalgebra isomorphism:
\begin{gather*}
c\colon \ \mathfrak{k}_\C \rightarrow \mathfrak{istr}(J_\C),\qquad (a,I,-a) \mapsto I + 2\imath L_a.
\end{gather*}
\end{Proposition}

\begin{proof}Expanding $\exp\big(\frac{\imath}{2}\ad\big(e_0^-\big)\big)\exp\big(\imath\ad\big(e_0^+\big)\big)$ we obtain
\begin{gather*}
c = 1 + \dfrac{\imath}{2}\ad\big(e_0^-\big)-\dfrac{1}{8}\ad\big(e_0^-\big)\ad\big(e_0^-\big)+\imath\ad\big(e_0^+\big) -\dfrac{1}{2}\ad\big(e_0^+\big)\ad\big(e_0^+\big)-\dfrac{1}{2}\ad\big(e_0^-\big)\ad\big(e_0^+\big)\\
\hphantom{c =}{} -\dfrac{\imath}{4}\ad\big(e_0^-\big)\ad\big(e_0^+\big)\ad\big(e_0^+\big) -\dfrac{\imath}{8}\ad\big(e_0^-\big)\ad\big(e_0^-\big)\ad\big(e_0^+\big)\\
\hphantom{c =}{} +\dfrac{1}{16}\ad\big(e_0^-\big)\ad\big(e_0^-\big)\ad\big(e_0^+\big)\ad\big(e_0^+\big).
\end{gather*}
We have the following straightforward calculations
\begin{gather*}
c(a,0,0) = a^- + 0+ 0 +2\imath L_a + a^+ - a^- - \imath L_a +0 + \dfrac{a^-}{4}= \left(\dfrac{a}{4}, \imath L_a, a\right),
\\
c(0,L_a+I,0) = (L_a+I) + \imath\dfrac{a^-}{2} +0 -\imath a^+ +0 - L_a +0 - \imath\dfrac{a^-}{4} +0= \left(\imath \dfrac{a}{4}, I, -\imath a\right),
\\
c(0,0,a) = a^+ -\imath L_a + \dfrac{a^-}{4} +0 +0 +0 +0+ 0+0 = \left(\dfrac{a}{4}, -\imath L_a, a\right),
\end{gather*}
proving the theorem.
\end{proof}

We complexify the Schr\"odinger representation $\pi$ given in Section \ref{SSSchrodRep} to obtain a representation $\pi_\C$ of $\mathfrak{g}_\C = J^-_\C\oplus \mathfrak{istr}(J_\C)\oplus J^+_\C$ acting on $\Fock$.
Explicitly, $\pi_\C$ is given by
\begin{itemize}\itemsep=0pt
\item $\pi_{\C} (e_l, 0, 0) = -2\imath z_l $,
\item $\pi_{\C} (0, L_{e_k}, 0) = -z_0\pt k + z_k\pt 0$,
\item $\pi_{\C} (0, [L_{e_i}, L_{e_j}], 0) = z_i\pt j - (-1)^{|i||j|}z_j\pt i$,
\item $\pi_{\C} (0, L_{e_0}, 0) = \frac{2-M}{2}-\E$,
\item $\pi_{\C} (0, 0, e_l) = -\frac{1}{2}\imath \mbessel(z_l)$,
\end{itemize}
with $i,j,k\in\{1, 2, \ldots, m+2n-1\}$ and $l\in\{0, 1, 2, \ldots, m+2n-1\}$. Since $L_{ij}$, $\E$ and $\mbessel(z_l)$ map~$\big\langle R^2\big\rangle $ into $\big\langle R^2\big\rangle $ this representation is well defined on $\Fock$. As in the bosonic case, we will define the \textit{Fock representation} $\rho$ as the composition of $\pi_{\C}$ with the Cayley transform $c$,
\begin{gather*}
\rho := \pi_{\C} \circ c.
\end{gather*}
So $\rho$ of $\mf{g} = J^-\oplus \mathfrak{istr}(J)\oplus J^+$ acting on $\Fock$ is given as follows
\begin{itemize}\itemsep=0pt
\item $\rho(e_0, 0, 0) = -\dfrac{\imath}{2} \big(z_0+\mbessel(z_0)+M-2+2\E\big)$,
\item $\rho(e_k, 0, 0) = -\dfrac{\imath}{2} \big(z_k+\mbessel(z_k)+2(z_0\pt k - z_k\pt 0)\big)$,
\item $\rho(0, L_{e_0}, 0) = \dfrac{1}{2}\big(z_0-\mbessel(z_0)\big)$,
\item $\rho(0, L_{e_k}, 0) = \dfrac{1}{2}\big(z_k-\mbessel(z_k)\big)$,
\item $\rho(0, [L_{e_i}, L_{e_j}], 0) = z_i\pt j - (-1)^{|i||j|}z_j\pt i$,
\item $\rho(0, 0, e_k) = -\dfrac{\imath}{2} \big(z_k+\mbessel(z_k)-2(z_0\pt k - z_k\pt 0)\big)$,
\item $\rho(0, 0, e_0) = -\dfrac{\imath}{2} \big(z_0+\mbessel(z_0)+2-M-2\E\big)$,
\end{itemize}
with $i,j,k\in\{1, 2, \ldots, m+2n-1\}$.

\begin{Proposition}[$\mf{osp}(m,2|2n)$-invariance]\label{PropSkewSymRho}
The Fock representation $\rho$ on $\Fock$ is skew-super\-symmetric with respect to the Bessel--Fischer product, i.e.,
\begin{align*}
\bfip{\rho(X)p,q} = - (-1)^{|X||p|}\bfip{p,\rho(X)q},
\end{align*}
for all $X\in \mathfrak{g}$ and $p,q\in \Fock$.
\end{Proposition}

\begin{proof}
Suppose $p,q\in \Fock$. From Corollary \ref{PropLij} we obtain
\begin{gather*}
\bfip{\rho(e_i,0,-e_i)p,q} = \bfip{-2\imath L_{0i}p,q}= -(-1)^{|i||p|}\bfip{p,2\overline\imath L_{0i}q} \\
\hphantom{\bfip{\rho(e_i,0,-e_i)p,q}}{}
 = -(-1)^{|i||p|}\bfip{p,\rho(e_i,0,-e_i)q}
\end{gather*}
and
\begin{gather*}
\bfip{\rho(0, [L_{e_i}, L_{e_j}], 0)p,q} = \bfip{ L_{ij}p,q} = -(-1)^{(|i|+|j|)|p|}\bfip{p,L_{ij}q} \\
\hphantom{\bfip{\rho(0, [L_{e_i}, L_{e_j}], 0)p,q}}{} = - (-1)^{(|i|+|j|)|p|}\bfip{p,\rho(0, [L_{e_i}, L_{e_j}], 0)q},
\end{gather*}
for $i,j \in \{1,\ldots, m+2n-1\}$.
Furthermore, we have
\begin{gather*}
\bfip{\rho(e_k,0,e_k)p,q} = \big\langle {-}\imath \big(z_k+\mbessel(z_k)\big)p,q\big\rangle_{\mathcal B}
 = -\imath\big( \bfip{z_kp,q}+\big\langle\mbessel(z_k)p,q\big\rangle_{\mathcal B}\big) \\
\hphantom{\bfip{\rho(e_k,0,e_k)p,q}}{} = -\imath(-1)^{|k||p|}\big( \big\langle p,\mbessel(z_k)q\big\rangle_{\mathcal B}+\bfip{p,z_k q}\big)\\
\hphantom{\bfip{\rho(e_k,0,e_k)p,q}}{}
 = -(-1)^{|k||p|}\big\langle p,\overline{\imath} \big( z_k+\mbessel(z_k)\big)q\big\rangle_{\mathcal B} \\
\hphantom{\bfip{\rho(e_k,0,e_k)p,q}}{}
 = - (-1)^{|k||p|}\bfip{p,\rho(e_k,0,e_k)q}
\end{gather*}
and similarly
\begin{gather*}
\bfip{\rho(0,L_{e_k},0)p,q} = \left\langle \dfrac{1}{2}\big( z_k-\mbessel(z_k)\big)p,q\right\rangle_{\mathcal B} = - (-1)^{|k||p|}\bfip{p,\rho(0,L_{e_k},0)q},
\end{gather*}
for $k \in \{0,\ldots, m+2n-1\}$. Because of Proposition~\ref{PropOrthog} we may assume $p$ and $q$ are homogeneous polynomials the same degree. We now have
\begin{align*}
\bfip{\rho(e_0,0,-e_0)p,q}& = \bfip{-2\imath (M-2+2\E)p,q}= -\bfip{p,-2\imath (M-2+2\E)q} \\
& = -\bfip{p,\rho(e_0,0,-e_0)q},
\end{align*}
which proves the theorem.
\end{proof}

\subsection[The $(\mf g, \mf k)$-module $F$]{The $\boldsymbol{(\mf g, \mf k)}$-module $\boldsymbol{F}$}\label{Section module F}
We define
\begin{gather*}
F := U(\mathfrak{g})1 \mod \big\langle R^2\big\rangle ,
\end{gather*}
where the $\mf g$ module structure is given by the Fock representation $\rho$. We also introduce the notation
\begin{gather*}
F_k := \bigoplus_{l=0}^k z_0^l\mathcal{H}_{k-l}\big(\C^{m-1|2n}\big) \mod \big\langle R^2\big\rangle
\end{gather*}
and reintroduce $r^2$ from Section \ref{SSRadSub} as its complexified version:
\begin{align*}
r^2 := \sum_{i=1}^{m+2n-1}z^i z_i.
\end{align*}
We have $R^2 = -z_0^2+r^2$ and since we are working modulo $\big\langle R^2\big\rangle $ this implies $z_0^2 = r^2$. In the following, we will work modulo $\big\langle R^2\big\rangle $ but omit $\big\langle R^2\big\rangle $ from our notation. For $M-1\not\in -2\N$ we now find
\begin{align*}
F_k &= \bigoplus_{l=0}^k z_0^l\mathcal{H}_{k-l}\big(\C^{m-1|2n}\big)
 = \bigoplus_{l=0}^{\big\lfloor\frac{k}{2}\big\rfloor}z_0^{2l}\mathcal{H}_{k-2l}\big(\C^{m-1|2n}\big) \oplus \bigoplus_{l=0}^{\big\lfloor\frac{k+1}{2}\big\rfloor-1}z_0^{2l+1}\mathcal{H}_{k-2l-1}\big(\C^{m-1|2n}\big)\\
&= \bigoplus_{l=0}^{\big\lfloor\frac{k}{2}\big\rfloor}r^{2l}\mathcal{H}_{k-2l}\big(\C^{m-1|2n}\big) \oplus \bigoplus_{l=0}^{\big\lfloor\frac{k+1}{2}\big\rfloor-1}z_0 r^{2l}\mathcal{H}_{k-2l-1}\big(\C^{m-1|2n}\big)\\
&\cong \mc P_k(\C^{m-1|2n})\oplus z_0\mc P_{k-1}\big(\C^{m-1|2n}\big),
\end{align*}
where we made use of the Fischer decomposition (Theorem~\ref{LemFC}) in the last step. In particular, by Proposition~\ref{Prop dimension},
\begin{gather*}
\dim F_k = \dim \mc P_k\big(\C^{m-1|2n}\big) + \dim\mc P_{k-1}\big(\C^{m-1|2n}\big) = \dim \mc H_k\big(\C^{m|2n}\big).
\end{gather*}
If $M\not\in -2\N$, then for $p\in F_k$ the Fischer decomposition on $ \mc P_k\big(\C^{m|2n}\big)$ also gives
\begin{gather*}
p = \sum_{l=0}^{\left\lfloor\frac{k}{2}\right\rfloor} R^{2l} h_{k-2l} = h_{k} \mod \big\langle R^2\big\rangle ,
\end{gather*}
with $h_{k-2l}\in \mc H_{k-2l}\big(\C^{m|2n}\big)$. This implies $F_k\cong \mc H_{k}\big(\C^{m|2n}\big)$ for $M\geq 2$.

\begin{Theorem}[decomposition of $F$]\label{ThDecF}
For $M\geq 3$, we have the following:
\begin{itemize}\itemsep=0pt
\item[$(1)$]$F_k$ is an irreducible $\mathfrak{k}$-module.
\item[$(2)$]$F$ is an irreducible $\mathfrak{g}$-module and its $\mathfrak{k}$-type decomposition is given by
\begin{gather*}
F = \bigoplus_{k=0}^\infty F_k.
\end{gather*}
\item[$(3)$]An explicit decomposition of $F_k$ into irreducible $\mathfrak{k}_{0}$-modules is given by
\begin{gather*}
F_k = \bigoplus_{l=0}^k z_0^l\mathcal{H}_{k-l}\big(\C^{m-1|2n}\big) \mod \big\langle R^2\big\rangle .
\end{gather*}
\end{itemize}
\end{Theorem}

\begin{proof}We have the following elements of the action of $\mf g$:
\begin{gather*}
\rho^+_0 := \rho\left(c^{-1}\left(\frac{-e_0}{2},0,0\right)\right) =\imath z_0,\\
\rho^-_0 := \rho\big(c^{-1}(0,0,-2e_0)\big) = \imath\mbessel(z_0),\\
\rho_{0i} := \rho\left(-\frac{e_i}{2},0,\frac{e_i}{2} \right) = \imath L_{0i},\\
\rho_{ij} :=\rho(0,[L_{e_i},L_{e_j}],0) = L_{ij},
\end{gather*}
with $i,j\in\{1,\ldots, m+2n-1\}$.
The elements $\rho_{ij}$ for $i, j \in \{1,\ldots , m+2n-1\}$, $i\leq j$ give rise to an irreducible representation of $\mf k_{0}$ on $\mc H_k\big(\C^{m-1|2n}\big)$ as a result of Proposition~\ref{PropIrred}. These elements leave $r^2 = z_0^2$ invariant and therefore also leave powers of $z_0$ invariant, which proves~$(3)$.

 Again by Proposition~\ref{PropIrred} the elements $\rho_{ij}$ and $\rho_{0i}$ give rise to an irreducible representation of $\mf k$ on $\mc H_k\big(\C^{m|2n}\big)\cong F_k$, which proves~$(1)$. For the first two elements we have
\begin{gather*}
\rho^+_0\big(z_0^k\big) = \imath z_0^{k+1},\\
\rho^-_0\big(z_0^k\big) = \imath\mbessel(z_0) z_0^k= \imath k(M+2k-4)z_0^{k-1} - \imath k(k-1)z_0^{k-1} = \imath k(M+k-3)z_0^{k-1},
\end{gather*}
which shows that $\rho^+_0$ allows us to go to polynomials of higher degrees while $\rho^-_0$ allows us to go the other direction for $M\geq 3$. Therefore we obtain~$(2)$.
\end{proof}

The following isomorphism is a direct result of this theorem.

\begin{Corollary}Suppose $M\geq 3$ and let
\begin{gather*}
\Fock = \mathcal P\big(\C^{m|2n}\big) / \big\langle R^2\big\rangle
\end{gather*}
be the polynomial Fock space defined in Definition~{\rm \ref{DefFock}}. We have $\Fock \cong F$.
\end{Corollary}

\section{The Segal--Bargmann transform}\label{The Segal--Bargmann transform}
In this section we construct the Segal--Bargmann transform and show that it is an isomorphism from $W$ as defined in Section~\ref{SSSchrodRep} to $F$ as defined in Section~\ref{Section module F}. It will make use of the integral~$\int_W$ we defined in Definition~\ref{DefIntW}. This integral is only defined for $M\geq 4$. Therefore we will always assume $M\geq 4$ throughout this section.

\subsection{Definition and properties}

Let $\widetilde I_\alpha(t)$ be the I-Bessel function as introduced in Appendix~\ref{AppBes}. We define an entire function~$\B_\alpha$ on~$\C$ by
\begin{gather*}
\B_\alpha(t) := \Gamma\left(\dfrac{M}{2}-1\right)\widetilde I_{\frac{M}{2}-2+\alpha}\big(2\sqrt{t}\big) = \Gamma\left(\dfrac{M}{2}-1\right)\sum_{l=0}^\infty \dfrac{1}{l! \Gamma\big(l+\frac{M}{2}-1+\alpha\big)}t^l.
\end{gather*}
Clearly we have $\B_0(0)= 1$ and
\begin{gather*}
\pt t^j\B_0(t) = \Gamma\left(\dfrac{M}{2}-1\right)\pt t^j \left(\widetilde I_{\frac{M}{2}-2}\big(2\sqrt{t}\big)\right) = \Gamma\left(\dfrac{M}{2}-1\right)\widetilde I_{\frac{M}{2}-2+j}\big(2\sqrt{t}\big) = \B_{j}(t).
\end{gather*}
We are now able to state the Segal--Bargmann transform that extends the one from the bosonic case obtained in~\cite{HKMO}.
\begin{Definition}\label{DefSB}
For $f\in W$ the \textit{Segal--Bargmann transform} is defined as
\begin{gather*}
\SB f(z) := \exp(-z_0)\int_W \B_0(x|z)\exp(-2x_0)f(x),
\end{gather*}
where $x|z$ is defined as
\begin{gather*}
x|z := 2x_0z_0+2\sum_{i,j=1}^{m+2n-1}x_i\beta^{ij}z_j
\end{gather*}
and we view $\B_\alpha(x|z)$ as a radial superfunction in the sense of equation~\eqref{EqDefRadSup2}.
\end{Definition}

Note that $\B_0(4(x|z))$ is the reproducing kernel $\mathbb {K}(x,z)$ of the Fock space we found in Theo\-rem~\ref{Theorem repr kernel}.

\begin{Proposition}For $M\geq 4$ the Segal--Bargmann transform $\SB$ is well defined.
\end{Proposition}

\begin{proof}We wish to prove that the integral
\begin{gather*}
\int_W \B_0(x|z)\exp(-2x_0)f(x)
\end{gather*}
is convergent for all $f\in W$ and that $\SB(R^2)=0$. As shown in the proof of Theorem~8.13 in~\cite{BF}, the elements of $W$ can be decomposed into elements of the form
\begin{gather*}
P_k \widetilde{K}_{-\frac{1}{2}+\alpha_1+\alpha_2}(2|X|),
\end{gather*}
with $P_k$ a homogeneous polynomial of degree $k$. Here $\widetilde K_\alpha$ is the K-Bessel function introduced in Appendix~\ref{AppBes} interpreted as a radial superfunction as in Section~\ref{SSRadSub}. Furthermore $\alpha_1, \alpha_2\in \N$ are subject to the relations $k \geq \alpha_1+2\alpha_2$ and $M\geq 2\alpha_1+2$. Also, observe that
\begin{gather}\label{EqXx_0}
|X| = \sqrt{\dfrac{x_0^2 + r^2}{2}} = \sqrt{x_0^2+\dfrac{R^2}{2}} = \sqrt{x_0^2}\quad \mod \big\langle R^2\big\rangle ,
\end{gather}
which is equal to $x_0$ within the domain of integration of $\int_W$. Because of all this and equation~\eqref{Eq Bessel is exp} it suffices to prove
\begin{gather*}
\int_W P_k\B_0(x|z)\widetilde{K}_{-\frac{1}{2}}(2|X|)\widetilde{K}_{-\frac{1}{2}+\alpha_1+\alpha_2}(2|X|)
\end{gather*}
is convergent for all $k\in \N$ and $\alpha_1, \alpha_2 \in \N$ subject to the above mentioned relations. We will use the explicit description of $\int_W$ given in equation~\eqref{EqExplW}. The morphism $\phi^\sharp$ leaves the degree of a polynomial unchanged. Hence, we can expand
\begin{gather*}
\big(\phi^\sharp(P_k)\big)_{|s=x_0=\rho} = \sum_{j=0}^k\rho^{k-j}a_j(\theta)b_j(\omega),
\end{gather*}
where $a_j(\theta)$ is a polynomial in $\mc P\big(\R^{0|2n}\big)$ of degree $j$ and $b_j(\omega)$ is a function depending on the spherical coordinates $\omega$. For $c\in \mathbb Z$ we obtain
\begin{gather*}
(1+\eta)^{c}= \sum_{j=0}^n \dfrac{1}{j!}\left(\dfrac{-c}{2}\right)_j\dfrac{\theta^{2j}}{2^j s^{2j}}, \qquad (1+\xi)^{c} = \sum_{j=0}^n \dfrac{(-1)^j}{j!}\left(\dfrac{-c}{2}\right)_j\dfrac{\theta^{2j}}{2^j x_0^{2j}}
\end{gather*}
and
\begin{gather*}
\phi^\sharp\big(\widetilde{K}_{\alpha}(c|X|)\big) = \widetilde{K}_{\alpha}(c|X|) = \sum_{j=0}^n \dfrac{(-1)^jc^{2j}\theta^{2j}}{j!8^j}\widetilde{K}_{\alpha+j}(c\rho),
\end{gather*}
from the proof of Lemma~8.6 in~\cite{BF}. Here $\eta$, $\xi$ and $\theta^2$ were defined in Section~\ref{SSIntC}. We introduce the notations
\begin{gather*}
\omega|z := 2\sum_{i=1}^{m-1}\omega_i z_i \qquad \mbox{and} \qquad \theta|z := 2 \sum_{i,j=m}^{m+2n-1} x_i\beta^{ij}z_j.
\end{gather*}
We can use equation~\eqref{EqDefRadSup2} with $h = \B_0$ and $f= x|z$ as a function in the $x$ variables to obtain
\begin{gather*}
\B_0(x|z) = \B_0(2 x_0 z_0 + s\omega|z + \theta|z) = \sum_{j=0}^{2n}\dfrac{1}{j!}(\theta|z)^j \B_j(2 x_0 z_0 + s\omega|z).
\end{gather*}
Using the properties of $\phi^\sharp$ described in Lemma~8.3 of~\cite{BF} and the expansion of $(1+\xi)$ and $(1+\eta)$, we now find
\begin{gather*}
\phi^\sharp(\B_0(x|z))_{|s=x_0=\rho} = \sum_{l_1=0}^{2n}\dfrac{1}{l_1!}(\theta|z)^{l_1} \B_{l_1}(2\rho(1+\xi)z_0 + \rho(1+\eta)\omega|z)_{|s=x_0=\rho}\\
\hphantom{\phi^\sharp(\B_0(x|z))_{|s=x_0=\rho}}{}
= \sum_{l_1=0}^{2n}\dfrac{1}{l_1!}(\theta|z)^{l_1} \B_{l_1}\left(\rho\sum_{l_2=0}^n\dfrac{1}{l_2!}\left(-\dfrac{1}{2}\right)_{l_2}\dfrac{\theta^{2l_2}}{2^{l_2}\rho^{2l_2}}(2(-1)^{l_2} z_0+ \omega|z)\right).
\end{gather*}
We use equation~\eqref{EqDefRadSup2} again, this time with $h=\B_{l_1}$ and $f$ equal to the sum over $l_2$. Note that the $l_2=0$ term corresponds with $f_0$. We obtain
\begin{gather*}
\phi^\sharp(\B_0(x|z))_{|s=x_0=\rho}= \sum_{l_1,l_3=0}^{2n}\dfrac{1}{l_1!l_3!}(\theta|z)^{l_1}\left(\rho\sum_{l_2=1}^n\dfrac{1}{l_2!}\left(-\dfrac{1}{2}\right)_{l_2}\dfrac{\theta^{2l_2}}{2^{l_2}\rho^{2l_2}}(2(-1)^{l_2} z_0+ \omega|z)\right)^{l_3}\\
 \hphantom{\phi^\sharp(\B_0(x|z))_{|s=x_0=\rho}=}{} \times \B_{l_1+l_3}(\rho(2z_0 +\omega|z)).
\end{gather*}

Combining all this, we see that
\begin{gather*}
 \dfrac{1}{\gamma} \int_0^\infty\int_{\mathbb S^{m-2}}\int_B\rho^{m-3}(1+\eta)^{m-3}(1+\xi)^{-1}\\
 \qquad{}\times \phi^\sharp(P_k\B_0(x|z)\widetilde{K}_{-\frac{1}{2}}(2|X|)\widetilde{K}_{-\frac{1}{2}+\alpha_1+\alpha_2}(|X|))_{|s=x_0=\rho}\mathrm{d}\rho \mathrm{d}\omega,
\end{gather*}
converges if
\begin{gather*}
 \int_0^\infty\int_B \rho^{m-3}\sum_{j_1=0}^k\sum_{j_2, j_3, j_4, j_5=0}^n\sum_{l_1,l_3=0}^{2n} \dfrac{1}{j_2!}\left(\dfrac{3-m}{2}\right)_{j_2}\dfrac{\theta^{2j_2}}{2^{j_2} \rho^{2{j_2}}} \dfrac{(-1)^{j_3}}{j_3!}\left(\dfrac{1}{2}\right)_{j_3}\dfrac{\theta^{2{j_3}}}{2^{j_3} \rho^{2{j_3}}}\\
 \qquad{} \times\rho^{k-j_1}a_{j_1}(\theta) \dfrac{1}{l_1!l_3!}(\theta|z)^{l_1} \left(\rho\sum_{l_2=1}^n\dfrac{1}{l_2!}\left(-\dfrac{1}{2}\right)_{l_2}\dfrac{\theta^{2l_2}}{2^{l_2}\rho^{2l_2}}c_1\right)^{l_3} \B_{l_1+l_3}(c_2\rho) \\
 \qquad {} \times\dfrac{(-1)^{j_4}\theta^{2{j_4}}}{{j_4}!2^{j_4}}\widetilde{K}_{-\frac{1}{2}+{j_4}}(2\rho)\dfrac{(-1)^{j_5}\theta^{2{j_5}}}{{j_5}!2^{j_5}}\widetilde{K}_{-\frac{1}{2}+\alpha_1+\alpha_2+j_5}(2\rho)\mathrm{d}\rho
\end{gather*}
converges for all $c_1, c_2\in\C$. This in turn converges if
\begin{gather*}
 \int_0^\infty\int_B \rho^{m-3+k-j_1+l_3-2l_3l_2-2j_2-2j_3}a_{j_1}(\theta)(\theta|z)^{l_1}\theta^{2j_2+2j_3+2j_4+2j_5+2l_3l_2} \B_{l_1+l_3}(c\rho)\\
 \qquad \times \widetilde{K}_{-\frac{1}{2}+{j_4}}(2\rho)\widetilde{K}_{-\frac{1}{2}+\alpha_1+\alpha_2+j_5}(2\rho)\mathrm{d}\rho
\end{gather*}
converges for all $0\leq j_1 \leq k$, $0\leq j_2, j_3, j_4, j_5 \leq n$, $0\leq l_1, l_3 \leq 2n$, $1\leq l_2 \leq n$ and all $c\in \C$. The Berezin integral is zero unless $j_1 + l_1 + 2j_2+2j_3+2j_4+2j_5+2l_3l_2 =2n$. The integral
\begin{gather*}
\int_0^\infty \rho^{\sigma -1}\widetilde I_{\beta_1}\big(\sqrt{a\rho}\big)\widetilde{K}_{\beta_2}(2\rho)\widetilde{K}_{\beta_3}(2\rho)\mathrm{d}\rho,
\end{gather*}
with $\beta_1 \geq 0$ converges if $\sigma > 2 \max\{\beta_2, 0\}+ 2 \max\{\beta_3, 0\}$. This follows from the asymptotic behaviour of the Bessel functions, see Appendix~\ref{AppBes}. Therefore we get the following condition
\begin{gather*}
m-2+k-j_1+l_3-2l_3l_2-2j_2-2j_3 \\
\qquad {}> 2 \max\left\lbrace -\frac{1}{2}+{j_4}, 0\right\rbrace
 + 2 \max\left\lbrace -\frac{1}{2}+\alpha_1+\alpha_2+j_5, 0\right\rbrace,
\end{gather*}
with $j_1 + l_1 + 2j_2+2j_3+2j_4+2j_5+2l_3l_2 =2n$. Taking into account $k\geq \alpha_1+2\alpha_2$ and $M\geq 2\alpha_1+2$ the condition reduces to $M > 2$. We still need that $\SB\big(R^2\big)=0$, but this follows easily from $\big(\phi^\sharp\big(R^2\big)\big)_{|s=x_0=\rho} = \big({-}x_0^2+s^2\big)_{|s=x_0=\rho}=0$.
\end{proof}

We can now show that $\SB$ intertwines the Schr\" odinger model with the Fock model.

\begin{Theorem}[intertwining property]\label{ThIP}
For $M\geq 4$ the Segal--Bargmann transform intertwines the action $\pi$ on $W$ with the action $\rho$ on $F$, i.e.,
\begin{gather*}
\SB\circ \pi(X) = \rho(X)\circ \SB,
\end{gather*}
for all $X\in \mathfrak{g}$.
\end{Theorem}

\begin{proof}The proof of this theorem is a technical and long but rather straightforward calculation. We refer to Appendix \ref{AppIP} for more details.
\end{proof}

\begin{Proposition} \label{Prop inverse SB transform}
The Segal--Bargmann transform $\SB$ induces a $\mf g$-module isomorphism between~$W$ and~$F$.
\end{Proposition}

\begin{proof}From the way we normalized the integral $\int_W$ and equation~(\ref{EqXx_0}) it is clear that
\begin{gather*}
\SB (\exp(-2|X|))(0) = \int_W \exp(-4|X|) = 1.
\end{gather*}
Therefore the Segal--Bargmann transform maps a non-zero element of $W$ to a non-zero element of~$F$. It also intertwines the actions $\pi$ and $\rho$. Since Theorems~\ref{ThDecW} and~\ref{ThDecF} give us that~$W$ and~$F$ are irreducible $\mf g$-modules, we conclude that $\SB$ is an isomorphism of $\mf g$-modules.
\end{proof}

\begin{Lemma}\label{Lemma SB constant}
We have $\SB(\exp(-2|X|)(z) = 1$.
\end{Lemma}
\begin{proof}
Since $\exp(-2|X|)$ is in $W_0$ Proposition \ref{Prop inverse SB transform} implies that $\SB (\exp(-2|X|))(z)$ is in $F_0=\C$. Hence $\SB (\exp(-2|X|))(z)$ is a constant. From the way we normalized the integral $\int_W$ we have $\SB (\exp(-2|X|))(0) = 1$ and therefore $\SB (\exp(-2|X|))(z) = 1$.
\end{proof}

\begin{Theorem}[unitary property]\label{PropUnitSB}
For $M\geq 4$ the Segal--Bargmann transform preserves the sesquilinear forms, i.e.,
\begin{gather*}
\bfip{\SB f,\SB g} = \ip{f, g},
\end{gather*}
for all $f,g\in W$.
\end{Theorem}

\begin{proof}
We first look at the case $f = \exp(-2|X|)$. Because of Lemma \ref{Lemma SB constant} and the superhermitian property of the Bessel--Fischer product we have
\begin{gather*}
\bfip{\SB(\exp(-2|X|)), \SB g} = \bfip{1, \SB g} = \bfipbar{\SB g, 1} = \SB(\overline{g(x)})\big(\mbessel\big)1\big|_{z=0}\\
 \hphantom{\bfip{\SB(\exp(-2|X|)), \SB g}}{} = \int_W \exp(-2x_0) \big( \B_0(x|\mbessel(z))\exp\big({-}\mbessel(z_0)\big) 1 \big) \overline{g(x)} \big|_{z=0},
\end{gather*}
for all $g\in W$. Here $\B_0\big(x|\mbessel(z)\big)$ and $\exp\big({-}\mbessel(z_0)\big)$ should be considered as infinite power sums of the Bessel operator with
\begin{gather*}
x|\mbessel(z) := 2x_0\mbessel(z_0)+2\sum_{i,j=1}^{m+2n-1}x_i\beta^{ij}\mbessel(z_j)=2\sum_{i,j=0}^{m+2n-1}x_i\beta^{ij}\bessel(z_j).
\end{gather*}
Since they act on a constant with respect to the variable $z$ we get
\begin{gather*}
\B_0\big(x|\mbessel(z)\big)\exp\big({-}\mbessel(z_0)\big)1 =\B_0(0)\exp(0)1=1,
\end{gather*}
which gives
\begin{gather*}
\bfip{\SB(\exp(-2|X|)), \SB g} = \int_W \exp(-2x_0)\overline{g(x)} = \ip{\exp(-2|X|), g},
\end{gather*}
if we use equation~(\ref{EqXx_0}). Now suppose $f, g\in W$. Since $W$ is an irreducible $\mf g$-module (Theorem~\ref{ThDecW}), there exists a $Y\in U(\mf g)$ such that $f= \pi(Y)\exp(-2|X|)$. Therefore we can reduce the general case to the previous case using the intertwining property (Theorem \ref{ThIP}) and the fact that the sesquilinear forms are skew symmetric for~$\pi$ and~$\rho$ (Propositions~\ref{PropSkewSymPi} and~\ref{PropSkewSymRho}):
\begin{gather*}
\bfip{\SB f, \SB g} = \bfip{\SB(\pi(Y)\exp(-2|X|)), \SB g} = \bfip{\rho(Y)\SB(\exp(-2|X|)), \SB g}\\
\hphantom{\bfip{\SB f, \SB g}}{} = - \bfip{\SB(\exp(-2|X|)), \rho(Y)\SB g} = -\bfip{\SB(\exp(-2|X|)), \SB (\pi(Y)g)}\\
\hphantom{\bfip{\SB f, \SB g}}{}= -\ip{\exp(-2|X|), \pi(Y)g} = \ip{\pi(Y)\exp(-2|X|), g} = \ip{f,g},
\end{gather*}
which proves the theorem.
\end{proof}

\subsection{The inverse Segal--Bargmann transform}

\begin{Definition}\label{Def SB inv}
For $p\in \mc P\big(\C^{m|2n}\big)$ the \textit{inverse Segal--Bargmann transform} is defined as
\begin{gather*}
\iSB p(x) := \exp(-2|X|)\B_0\big(x|\mbessel(z)\big)\exp\big({-}\mbessel(z_0)\big)p(z)\big|_{z=0},
\end{gather*}
with
\begin{gather*}
x|\mbessel(z) := 2x_0\mbessel(z_0)+2\sum_{i,j=1}^{m+2n-1}x_i\beta^{ij}\mbessel(z_j)=2\sum_{i,j=0}^{m+2n-1}x_i\beta^{ij}\bessel(z_j).
\end{gather*}
\end{Definition}

Note that both $\B_0\big(x|\mbessel(z)\big)$ and $\exp\big({-}\mbessel(z_0)\big)$ are infinite power sums of the Bessel operator. However, they are well defined operators on polynomials in $z$ since the Bessel operator lowers the degree of the polynomials and therefore the power operators become zero after a finite number of terms. Thus $\iSB$ is a well-defined operator on $\mc P\big(\C^{m|2n}\big)$. Because of Proposition~\ref{PropTangBes} it maps $\big\langle R^2 \big\rangle $ to zero and thus $\iSB$ can be restricted to~$F$. Moreover, it is also well defined as the inverse of the Segal--Bargmann transform. This follows from the following proposition.

\begin{Proposition}\label{PropInv}The inverse Segal--Bargmann transform is well defined as the inverse of the Segal--Bargmann transform defined in Definition~{\rm \ref{DefSB}}.
\end{Proposition}

\begin{proof} From Proposition \ref{Prop inverse SB transform}, we know that $\SB$ has an inverse.
Suppose the operator $A$ is this inverse. Using Theorem \ref{PropUnitSB} we then have the following calculation:
\begin{align*}
\ip{A p,\psi} &= \bfip{\SB(A p),\SB\psi} = (-1)^{|p||\psi|}\bfipbar{\SB\psi,p} = (-1)^{|p||\psi|} \overline{\SB\psi}\big(\mbessel\big)p(z)\big|_{z=0}\\
&= (-1)^{|p||\psi|} \left.\left(\exp\big({-}\mbessel(z_0)\big)\int_W \B_0\big(x|\mbessel(z)\big)\exp(-2x_0)\overline{\psi(x)}p(z)\right)\right|_{z=0}\\
&= \int_W \big(\exp\big({-}\mbessel(z_0)\big) \B_0\big(x|\mbessel(z)\big)\exp(-2x_0)p(z)\big)\big|_{z=0}\overline{\psi(x)}\\
&= \big\langle \exp\big({-}\mbessel(z_0)\big) \B_0\big(x|\mbessel(z)\big)\exp(-2|X|)p(z)\big|_{z=0},\psi\big\rangle_W,
\end{align*}
where we used equation~(\ref{EqXx_0}) in the last step. Since the sesquilinear form $\ip{\cdot\, , \cdot}$ is non-degenerate we obtain $A=\SB^{-1}$.
\end{proof}

We can make the inverse Segal--Bargmann transform more explicit on the space of homogeneous polynomials. To the best of our knowledge, this explicit expression is also new for the bosonic case.

\begin{Proposition}\label{PropExplicitInv}
For $p\in \mc P_k\big(\C^{m|2n}\big)$, $k\in \N$ the inverse Segal--Bargmann transform $\iSB$ is given by
\begin{gather*}
\iSB p(x) = \exp(-2|X|) \sum_{j=0}^k \dfrac{(-1)^j}{j!(k-j)!}\dfrac{\Gamma\big(\frac{M}{2}-1\big)}{\Gamma\big(k-j+\frac{M}{2}-1\big)}\big\langle z_0^j(x|z)^{k-j}, \overline{p}\big\rangle.	
\end{gather*}
\end{Proposition}

\begin{proof}
For $p\in \mc P_k\big(\C^{m|2n}\big)$ we have
\begin{align*}
\iSB p(x) &= \exp(-2|X|)\B_0\big(x|\mbessel(z)\big)\exp\big({-}\mbessel(z_0)\big)p(z)\big|_{z=0}\\
&= \exp(-2|X|)\bfip{\B_0(x|z)\exp(-z_0), \overline{p}}.
\end{align*}
Because of the orthogonality of the Bessel--Fischer product we only need to look at the homogeneous polynomial term of degree $k$ in the expansion of $\B_0(x|z)\exp(-z_0)$. Explicitly, we need the homogeneous term of degree $k$ in
\begin{gather*}
\Gamma\left(\dfrac{M}{2}-1\right)\sum_{l=0}^\infty \dfrac{1}{l! \Gamma\big(l+\frac{M}{2}-1\big)}(x|z)^l\sum_{j=0}^\infty \dfrac{(-1)^j}{j!}z_0^j,
\end{gather*}
which is
\begin{gather*}
\Gamma\left(\frac{M}{2}-1\right)\sum_{j=0}^k \dfrac{1}{(k-j)!\Gamma\big(k-j+\frac{M}{2}-1\big)}(x|z)^{k-j} \dfrac{(-1)^j}{j!}z_0^j
\end{gather*}
as desired.
\end{proof}

\subsection{The generalized Hermite functions}
As a standard application of the Segal--Bargmann transform, we can construct generalized Hermite functions which extend the ones of the bosonic case given in~\cite{HKMO}.

\begin{Definition}
The \textit{generalized Hermite functions on $W$} are defined by
\begin{gather*}
h_\alpha(x) := \exp(2|X|)\left(\dfrac{1}{2}\mbessel\right)^\alpha\exp(-4|X|),
\qquad \text{with} \quad
\left(\dfrac{1}{2}\mbessel\right)^\alpha := \prod_i \dfrac{1}{2^{\alpha_i}}\mbessel(e_i)^{\alpha_i},
\end{gather*}
for $\alpha\in \N^{m|2n}$. The \textit{generalized Hermite polynomials} $H_\alpha$ are defined by the equation
\begin{gather*}
h_\alpha(x) = H_\alpha(x)\exp(-2|X|).
\end{gather*}
\end{Definition}

\begin{Proposition}[Hermite to monomial property] We have
\begin{gather*}
\SB h_\alpha = (2z)^\alpha.
\end{gather*}
\end{Proposition}

\begin{proof}
We will use the fact that $\int_W$ is supersymmetric with respect to the Bessel operators \cite[Proposition 8.9]{BF}:
\begin{gather*}
\int_W (\bessel(x_k)f)g = (-1)^{|f||k|} \int_W f(\bessel(x_k)g).
\end{gather*}

Combining this with equation~\eqref{EqB0I} of Lemma \ref{Properties of B_0} we find
\begin{align*}
\SB h_\alpha(z) &= \exp(-z_0)\int_W \B_0(x|z)\exp(-2x_0)h_\alpha(x)\\
&= \exp(-z_0)\int_W \B_0(x|z)\prod_i \dfrac{1}{2^{\alpha_i}}\mbessel(x_i)^{\alpha_i}\exp(-4|X|)\\
&= \exp(-z_0)\int_W \prod_i \dfrac{1}{2^{\alpha_i}}\mbessel(x_i)^{\alpha_i}\left(\B_0(x|z)\right)\exp(-4|X|)\\
&= \exp(-z_0)\int_W \prod_i\dfrac{1}{2^{\alpha_i}} (4z_i)^{\alpha_i}\left(\B_0(x|z)\right)\exp(-4|X|)\\
&= \exp(-z_0)\int_W (2z)^\alpha\B_0(x|z)\exp(-4|X|)\\
&= (2z)^\alpha \SB(\exp(-2|X|))(z).
\end{align*}
Because of Lemma \ref{Lemma SB constant} the theorem follows.
\end{proof}

\appendix
\section{Special functions}\label{Special functions}

\subsection{Bessel functions}\label{AppBes}
The I-Bessel function $I_\alpha(t)$ (or modified Bessel function of the first kind) is defined by
\begin{gather*}
I_\alpha(t) := \left(\dfrac{t}{2}\right)^{\alpha}\sum_{k=0}^\infty \dfrac{1}{k!\Gamma(k+\alpha+1)}\left(\dfrac{t}{2}\right)^{2k}
\end{gather*}
and the K-Bessel function $K_\alpha$ (or modified Bessel function of the third kind) by
\begin{gather*}
K_\alpha(t) := \dfrac{\pi}{2\sin(\pi \alpha)}(I_{-\alpha}(t)-I_\alpha(t)),
\end{gather*}
for $\alpha, t\in\C$, see \cite[Section~4.12]{AAR}. In this paper we will need the following renormalisations
\begin{gather*}
\widetilde I_\alpha(t) := \left(\dfrac{t}{2}\right)^{-\alpha} I_\alpha(t), \qquad \widetilde K_\alpha(t) := \left(\dfrac{t}{2}\right)^{-\alpha} K_\alpha(t).
\end{gather*}
Remark that we have the following special case \cite[equation~(4.12.4)]{AAR}
\begin{gather} \label{Eq Bessel is exp}
\widetilde K_{-\frac{1}{2}}(t) = \frac{\sqrt{\pi}}{2}\exp(-t).
\end{gather}

The asymptotic behaviour of the I-Bessel function can be deducted from equations (4.12.7) and (4.12.8) of \cite{AAR},
\begin{gather*}
\widetilde I_\alpha(t) = \dfrac{1}{\sqrt{2\pi}}\big({-}t^2\big)^{-\frac{2\alpha+1}{4}}\left(\exp\left(-i\left(\dfrac{(2\alpha+1)\pi}{4}-\sqrt{-t^2}\right)\right)\left(1+\mc O\left(\dfrac{1}{t}\right)\right)\right.\\
 \left. \hphantom{\widetilde I_\alpha(t) =}{} + \exp\left(i\left(\dfrac{(2\alpha+1)\pi}{4}-\sqrt{-t^2}\right)\right)\left(1+\mc O\left(\dfrac{1}{t}\right)\right)\right),
\end{gather*}
for $|t|\rightarrow +\infty$. The asymptotic behaviour of the K-Bessel function for $t\in \R$ is given in Appendix~B.2 of~\cite{BF},
\begin{gather*}
\mbox{for }t\rightarrow 0\colon \quad \widetilde K_\alpha(t) = \begin{cases}
\displaystyle \frac{\Gamma(\alpha)}{2}\left(\frac{t}{2}\right)^{-2\alpha} + o\big(t^{-2\alpha}\big) & \mbox{if }\alpha >0,\vspace{1mm}\\
\displaystyle -\log\left(\frac{t}{2}\right)+o\left(\log\left(\frac{t}{2}\right)\right) & \mbox{if }\alpha =0,\vspace{1mm}\\
\displaystyle \frac{\Gamma(-\alpha)}{2}+o(1) & \mbox{if }\alpha <0,
\end{cases}\\
\mbox{for }t\rightarrow +\infty\colon \quad \widetilde K_\alpha(t) = \dfrac{\sqrt{\pi}}{2}\left(\dfrac{t}{2}\right)^{-\alpha-\frac{1}{2}}e^{-t}\left(1+\mc O\left(\dfrac{1}{t}\right)\right).
\end{gather*}

\subsection{Generalised Laguerre functions}\label{AppLag}
Consider the generating function
\begin{gather*}
G_2^{\mu,\nu}(t,x) := \dfrac{1}{(1-t)^{\frac{\mu+\nu+2}{2}}}\widetilde I_{\frac{\mu}{2}}\left(\dfrac{tx}{1-t}\right)\widetilde K_{\frac{\nu}{2}}\left(\dfrac{x}{1-t}\right),
\end{gather*}
for complex parameters $\mu$ and $\nu$. The generalised Laguerre functions $\Lambda_{2,j}^{\mu,\nu}(x)$ are defined in~\cite{HKMM} as the coefficients in the expansion
\begin{gather*}
G_2^{\mu,\nu}(t,x) = \sum_{j=0}^{\infty}\Lambda_{2,j}^{\mu,\nu}(x)t^j.
\end{gather*}

\section{Proof of Theorem \ref{ThIP}}\label{AppIP}

To give the proof of Theorem \ref{ThIP} we first need a few technical lemmas.

\begin{Lemma}[properties of $\B_0$]\label{Properties of B_0} For $k\in\{1, \ldots, m+2n-1\}$ we have
\begin{alignat}{3} \label{EqZI}
& \pt {z^k} \B_0(x|z)= 2x_k \B_1(x|z), \qquad && \pt {x^k} \B_0(x|z)= 2z_k \B_1(x|z),&\\
\label{EqEI}
& \pt {z^0} \B_0(x|z)= -2x_0 \B_1(x|z), \qquad && \pt {x^0} \B_0(x|z)= -2z_0 \B_1(x|z),&\\
 & \E_z \B_0(x|z)= (x|z)\B_1(x|z), \qquad && \E_x \B_0(x|z)= (x|z)\B_1(x|z),&\nonumber
\end{alignat}
and for $i\in\{0, \ldots, m+2n-1\}$ we have
\begin{gather}\label{EqB0I}
\mbessel(z_i)\B_0(x|z) = 4x_i\B_0(x|z), \qquad \mbessel(x_i)\B_0(x|z) = 4z_i\B_0(x|z).
\end{gather}
The last equation expresses that $\B_0(x|z)$ is an eigenfunction of $\mbessel$.
\end{Lemma}

\begin{proof}
Equation~\eqref{EqZI} follows from the chain rule. Equation \eqref{EqEI} follows immediately from equation~\eqref{EqZI} and the definition of the Euler operator.
Using the same calculation as in the proof of Lemma \ref{Lemma rep kernel}, we obtain \begin{gather*}
\mbessel(z_k)(x|z)^l=4x_kl\left(l+\frac{M}{2}-2\right)(x|z)^{l-1}.
\end{gather*}
We then find
\begin{align*}
\mbessel(z_k)\B_0(x|z) &= \Gamma\left(\dfrac{M}{2}-1\right)\sum_{l=0}^\infty \dfrac{1}{l!\,\Gamma\left(l+\frac{M}{2}-1\right)}\mbessel(z_k)(x|z)^l,\\
 &= 4x_k\Gamma\left(\dfrac{M}{2}-1\right)\sum_{l=1}^\infty \dfrac{l\left(l+\frac{M}{2}-2\right)}{l!\,\Gamma\left(l+\frac{M}{2}-1\right)}(x|z)^{l-1}\\
&= 4x_k\Gamma\left(\dfrac{M}{2}-1\right)\sum_{l-1=0}^\infty \dfrac{1}{(l-1)!\,\Gamma\left((l-1)+\frac{M}{2}-1\right)}(x|z)^{l-1}\\
&= 4x_k\B_0(x|z).
\end{align*}
The calculations for the case $\mbessel(x_k)\B_0(x|z)$ is analogous.
\end{proof}

\begin{Lemma}[properties of $\exp$] For $k\in \{1, \ldots, m+2n-1\}$ we have
\begin{alignat}{3}\label{EqEexp}
& \E_z \exp(-z_0)=-z_0\exp(-z_0), \qquad && \E_x \exp(-2x_0)= -2x_0\exp(-2x_0),&\\\label{EqBzkexp}
&\mbessel(z_k)\exp(-z_0)= z_k \exp(-z_0), \qquad && \mbessel(x_k)\exp(-2x_0)= 4x_k\exp(-2x_0),&\\
&\label{EqBz0exp}
\mbessel(z_0)\exp(-z_0)=(2-M+z_0)\exp(-z_0).\quad &&&
\end{alignat}
\end{Lemma}
\begin{proof} Equation \eqref{EqEexp} is immediate. For equation~\eqref{EqBz0exp} we have
\begin{align*}
\mbessel(z_0)\exp(-z_0) &= (2-M-2\E)\pt 0\exp(-z_0) + z_0\Delta\exp(-z_0)\\
&= (2-M-2\E)\exp(-z_0) - z_0\pt 0\pt 0 \exp(-z_0)\\
&= (2-M+2z_0)\exp(-z_0) - z_0 \exp(-z_0)\\
&= (2-M+z_0)\exp(-z_0),
\end{align*}
while for equation~\eqref{EqBzkexp} we compute
\begin{gather*}
\mbessel(z_k)\exp(-z_0) = 0 - z_k\Delta\exp(-z_0)= z_k\pt 0\pt 0 \exp(-z_0) = z_k \exp(-z_0).
\end{gather*}
A similar calculation shows $\mbessel(x_k)\exp(-2x_0) = 4x_k\exp(-2x_0)$.
\end{proof}

\begin{Lemma}[properties of $\SB$] We have
\begin{gather}\label{EqEuler}
\E_z \SB f(z) = -z_0\SB f(z)+ \int_W \left(x|z\right)\B_1(x|z)\exp(-z_0-2x_0)f(x),\\
\dfrac{1}{2}\mbessel(z_0)\SB f(z) = \dfrac{1}{2}(2-M+z_0)\SB f(z) + 2\SB(x_0 f)(z)\nonumber\\
\hphantom{\dfrac{1}{2}\mbessel(z_0)\SB f(z) =}{} - \int_W (x|z )\B_1(x|z)\exp(-z_0-2x_0)f(x),\label{EqBessel0}\\
\dfrac{1}{2}\mbessel(z_k)\SB f(z) = \dfrac{1}{2}z_k\SB f(z) + 2\SB(x_k f)(z)\nonumber\\
\hphantom{\dfrac{1}{2}\mbessel(z_k)\SB f(z) =}{} - 2\int_W (z_0x_k+z_kx_0 )\B_1(x|z)\exp(-z_0-2x_0)f(x),\label{EqBesselk}\\
\label{EqImpuls}
L_{0k}^z \SB f(z) = -z_k\SB f(z) + 2\int_W \left(z_0x_k+z_kx_0\right)\B_1(x|z)\exp(-z_0-2x_0)f(x).
\end{gather}
\end{Lemma}

\begin{proof}Equation (\ref{EqEuler}) follows from~(\ref{EqEI}) and~(\ref{EqEexp}). To prove equation~(\ref{EqBessel0}) we
 apply the product rule (Proposition~\ref{PropProdRule}) to $-\mathcal{B}(z_0) (\exp(-z_0)\B_0(x|z))$ and substituting equations \eqref{EqB0I}, \eqref{EqBz0exp}, and the following easily verified identities
\begin{gather*}
\E_z(\exp(-z_0))\pt {z^0}(\B_0(x|z)) = 2z_0x_0\exp(-z_0)\B_1(x|z),\\
 \pt {z^0} (\exp(-z_0))\E_z(\B_0(x|z)) = \left(x|z\right)\exp(-z_0)\B_1(x|z),\\
z_0\sum_{r,s}\beta^{rs}\pt {z^r}(\exp(-z_0))\pt {z^s}(\B_0(x|z)) = 2z_0x_0\exp(-z_0)\B_1(x|z).
\end{gather*}
Equation (\ref{EqBesselk}) follows in the same way using equations (\ref{EqB0I}), (\ref{EqBzkexp}), and
\begin{gather*}
\E_z(\exp(-z_0))\pt {z^k}(\B_0(x|z)) =- 2z_0x_k\exp(-z_0)\B_1(x|z),\\
 \pt {z^k} (\exp(-z_0))\E_z(\B_0(x|z)) = 0,\\
z_k\sum_{r,s}\beta^{rs}\pt {z^r}(\exp(-z_0))\pt {z^s}(\B_0(x|z)) = 2z_kx_0\exp(-z_0)\B_1(x|z),
\end{gather*}
while equation~(\ref{EqImpuls}) follows immediately from
\begin{gather*}
L_{0k}^z (\exp(-z_0)\B_0(x|z) ) =L_{0k}^z (\exp(-z_0) )\B_0(x|z) + \exp(-z_0)L_{0k}^z (\B_0(x|z) )\\
\qquad{} = -z_k \exp(-z_0) \B_0(x|z) + \exp(-z_0) (z_0(2x_k\B_1(x|z))-z_k(-2x_0\B_1(x|z)) )\\
\qquad{} = -z_k\exp(-z_0) \B_0(x|z) + 2\exp(-z_0) (z_0x_k+z_kx_0 )\B_1(x|z).
\end{gather*}
This proves the lemma.
\end{proof}

We can now prove Theorem \ref{ThIP}. For convenience we restate it here.

\begin{Theorem}[intertwining property]For $M\geq 4$ the Segal--Bargmann transform intertwines the action $\pi$ on $W$ with the action $\rho$ on $F$, i.e.,
\begin{gather}\label{Eq Intertwining}
\SB\circ \pi(X) = \rho(X)\circ \SB,
\end{gather}
for all $X\in \mathfrak{g}$.
\end{Theorem}

\begin{proof}
We will use the decomposition $\mathfrak{g}=J^-\oplus \mf{istr}(J) \oplus J^+$ to prove equation~\eqref{Eq Intertwining} case by case.

\textbf{Case 1:} $X=(e_0,0,0)$. We wish to prove
\begin{gather*}
\SB(2x_0f)(z)=\dfrac{1}{2}z_0\SB f(z)+\dfrac{1}{2}\mbessel(z_0)\SB f(z)+\left(\dfrac{M}{2}-1\right)\SB f(z) + \E_z \SB f(z).
\end{gather*}
If we substitute (\ref{EqEuler}) and (\ref{EqBessel0}) into the equation the result follows.

\textbf{Case 2:} $X=(e_k,0,0)$, $k\neq 0$. Substituting \eqref{EqBesselk} and \eqref{EqImpuls} in
\begin{gather*}
\SB(2x_kf)(z)=\dfrac{1}{2}z_k\SB f(z)+\dfrac{1}{2}\mbessel(z_k)\SB f(z)+L_{0k}^z \SB f(z)
\end{gather*}
proves equation~\eqref{Eq Intertwining} for $X=(e_k,0,0)$.

\textbf{Case 3:} $X=(0,L_{e_0},0)$. We wish to prove
\begin{gather*}
\dfrac{1}{2}\SB ((2-M) f )-\SB (\E_x f )(z)=\dfrac{1}{2}z_0\SB f(z)-\dfrac{1}{2}\mbessel(z_0)\SB f(z).
\end{gather*}
Using (\ref{EqBessel0}) this becomes
\begin{gather*}
\SB ((2-M) f ) =\SB (\E_x f )(z) -2\SB(x_0 f)(z) + \int_W (x|z )\B_1(x|z)\exp(-z_0-2x_0)f(x).
\end{gather*}
From the proof of \cite[Proposition~8.9]{BF} it follows that $\int_W(\E_x+M-2)f =0$ when $\int_W f$ is well defined. This gives
\begin{align*}
\SB ((2-M) f ) &= \exp(-z_0)\int_W (2-M)\B_0(x|z)\exp(-2x_0)f(x)\\
&= \exp(-z_0)\int_W \E_x (\B_0(x|z)\exp(-2x_0)f(x) )\\
&= \int_W (x|z )\B_1(x|z)\exp(-z_0-2x_0)f(x) - 2\SB(x_0 f)(z) + \SB (\E_x f )(z),
\end{align*}
where we used (\ref{EqEI}) and (\ref{EqEexp}) to obtain the last equality.

\textbf{Case 4:} $X=(0,L_{e_k},0)$, $k\neq 0$. In a similar fashion to (\ref{EqImpuls}), we find
\begin{align*}
-\SB\big(L_{0k}^x f\big)(z) &= -\exp(-z_0)\int_W \B_0(x|z)\exp(-2x_0)L_{0k}^xf(x)\\
&= \exp(-z_0)\int_W L_{0k}^x (\B_0(x|z)\exp(-2x_0) )f(x)\\
&= 2\int_W (x_0z_k+z_0x_k )\B_1(x|z)\exp(-z_0-2x_0)f(x) - 2\SB(x_k f)(z).
\end{align*}
Because of (\ref{EqBesselk}), this is equivalent to
\begin{gather*}
-\SB\big(L_{0k}^x f\big)(z)=\dfrac{1}{2}z_k\SB f(z)-\dfrac{1}{2}\mbessel(z_k)\SB f(z),
\end{gather*}
proving equation~\eqref{Eq Intertwining} for $X=(0,L_{e_k},0)$.

\textbf{Case 5:} $X=(0,[L_{e_i},L_{e_j}],0)$, $i,j\neq 0$. Using Proposition \ref{PropSkewSymPi}, we obtain
\begin{align*}
\SB\big(L_{ij}^x f\big)(z)&=\exp(-z_0)\int_W \B_0(x|z)\exp(-2x_0)L_{ij}^xf(x)\\
&=-\exp(-z_0)\int_W L_{ij}^x (\B_0(x|z)\exp(-2x_0) )f(x)\\
&= -2\exp(-z_0)\int_W \big(x_iz_j-(-1)^{|i||j|}x_jz_i\big)\B_1(x|z)\exp(-2x_0)f(x)\\
&= \exp(-z_0)\int_W L_{ij}^z (\B_0(x|z) )\exp(-2x_0)f(x)\\
&=L_{ij}^z\SB f(z),
\end{align*}
thus equation~\eqref{Eq Intertwining} holds.

\textbf{Case 6:} $X=(0,0,e_0)$. Using (\ref{EqEuler}) and (\ref{EqBessel0}), we can rewrite
\begin{gather*}
\dfrac{1}{2}z_0\SB f(z)+\dfrac{1}{2}\mbessel(z_0)\SB f(z)+\left(1-\dfrac{M}{2}\right)\SB f(z) - \E_z \SB f(z)
\end{gather*}
as
\begin{gather*}
2z_0\SB f(z) + 2\SB(x_0 f)(z)+(2-M)\SB f(z) - 2\int_W (x|z )\B_1(x|z)\exp(-z_0-2x_0)f(x).
\end{gather*}
In a similar fashion to (\ref{EqBessel0}), we find
\begin{gather*}
\dfrac{1}{2}\SB\big(\mbessel(x_0)f\big)(z) = \dfrac{1}{2}\exp(-z_0)\int_W \B_0(x|z)\exp(-2x_0)\mbessel(x_0)f(x)\\
\hphantom{\dfrac{1}{2}\SB\big(\mbessel(x_0)f\big)(z)}{} =\dfrac{1}{2}\exp(-z_0)\int_W \mbessel(x_0)\left(\B_0(x|z)\exp(-2x_0)\right)f(x)\\
\hphantom{\dfrac{1}{2}\SB\big(\mbessel(x_0)f\big)(z)}{} = 2z_0\SB f(z) + (2-M)\SB f(z) +2 \SB(x_0 f)(z) \\
\hphantom{\dfrac{1}{2}\SB\big(\mbessel(x_0)f\big)(z)=}{} -2\int_W (x|z )\B_1(x|z)\exp(-z_0-2x_0)f(x).
\end{gather*}
So we conclude that equation~\eqref{Eq Intertwining} holds for $X=(0,0,e_0)$.

\textbf{Case 7:} $X=(0,0,e_k)$, $k\neq 0$. To show
\begin{gather*}
\dfrac{1}{2}\SB\big(\mbessel(x_k)f\big)(z)=\dfrac{1}{2}z_k\SB f(z)+\dfrac{1}{2}\mbessel(z_k)\SB f(z)-L_{0k}^z \SB f(z),
\end{gather*}
we substitute (\ref{EqBesselk}) and (\ref{EqImpuls}) into the equation. So we obtain
\begin{gather*}
\dfrac{1}{2}\SB\big(\mbessel(x_k)f\big)(z) = 2z_k\SB f(z) + 2\SB(x_k f)(z)\\
\hphantom{\dfrac{1}{2}\SB\big(\mbessel(x_k)f\big)(z) =}{} - 4\int_W (z_0x_k+z_kx_0 )\B_1(x|z)\exp(-z_0-2x_0)f(x).
\end{gather*}
In a similar fashion to (\ref{EqBesselk}), we find
\begin{gather*}
\dfrac{1}{2}\SB\big(\mbessel(x_k)f\big)(z) = \dfrac{1}{2}\exp(-z_0)\int_W \B_0 (x|z )\exp(-2x_0)\mbessel(x_k)f(x)\\
\hphantom{\dfrac{1}{2}\SB\big(\mbessel(x_k)f\big)(z)}{} =\dfrac{1}{2}\exp(-z_0)\int_W \mbessel(x_k) (\B_0(x|z)\exp(-2x_0) )f(x)\\
\hphantom{\dfrac{1}{2}\SB\big(\mbessel(x_k)f\big)(z)}{} = 2z_k\SB f(z) + 2\SB(x_k f)(z)\\
\hphantom{\dfrac{1}{2}\SB\big(\mbessel(x_k)f\big)(z)=}{} - 4\int_W (z_0x_k+z_kx_0 )\B_1(x|z)\exp(-z_0-2x_0)f(x).
\end{gather*}
This proves the theorem.
\end{proof}

\subsection*{Acknowledgements}
SB is supported by a BOF Postdoctoral Fellowship from Ghent University.
HDB is supported by the Research Foundation Flanders (FWO) under Grant EOS 30889451.
The authors would like to thank the anonymous referees for carefully reading the paper and for their useful comments.

\pdfbookmark[1]{References}{ref}
\LastPageEnding

\end{document}